\documentclass[a4paper, 10pt]{article}
\usepackage{lineno}
\usepackage{hyperref}
\usepackage{enumitem}
\usepackage{float}
\usepackage{amsmath, amsthm}
\usepackage{amsfonts}
\usepackage{amssymb}
\usepackage{fixltx2e}

\usepackage{subcaption}

\usepackage{enumitem}

\usepackage[hyperref, dvipsnames]{xcolor}
\hypersetup{
  colorlinks=true,
}
\usepackage{amsmath, amsthm, amssymb, color, mathtools, bm, mathrsfs, authblk}
\modulolinenumbers[5]
\makeatletter
\def\ps@pprintTitle{%
  \let\@oddhead\@empty
  \let\@evenhead\@empty
  \def\@oddfoot{\reset@font\hfil\thepage\hfil}
  \let\@evenfoot\@oddfoot
}
\makeatother

\newtheorem{theorem}{Theorem}[section]
\newtheorem{proposition}[theorem]{Proposition}
\newtheorem{lemma}[theorem]{Lemma}
\newtheorem{corollary}[theorem]{Corollary}

\newtheorem{remark}[theorem]{Remark}

\theoremstyle{definition}
\newtheorem{definition}{Definition}[section]

\usepackage{comment}
\newcommand{\grad}{\nabla}

\newcommand{\norm}[2]{\left\lVert #1\right\rVert_{#2}}
\newcommand{\weaklyto}{\rightharpoonup}

\newcommand{\Lip}[1]{\mathrm{Lip}(#1)}

\begin{document}
\hypersetup{
  urlcolor     = blue, 
  linkcolor    = Bittersweet, 
  citecolor   = Cerulean
}


\title{Directional differentiability for elliptic quasi-variational inequalities of obstacle type}

\author{Amal Alphonse, Michael Hinterm\"uller and Carlos N. Rautenberg}
\affil{Weierstrass Institute\\Mohrenstrasse 39\\10117 Berlin\\Germany}
\maketitle
\begin{abstract}
The directional differentiability of the solution map of obstacle type quasi-variational inequalities (QVIs) with respect to perturbations on the forcing term is studied.  The classical result of Mignot is then extended to the quasi-variational case under assumptions that allow multiple solutions of the QVI. The proof involves selection procedures for the solution set and represents the directional derivative as the limit of a monotonic sequence of directional derivatives associated to specific variational inequalities. Additionally, estimates on the coincidence set and several simplifications under higher regularity are studied. The theory is illustrated by a detailed study of an application to thermoforming comprising of modelling, analysis and some numerical experiments.\end{abstract}

\maketitle
\section{Introduction}
Quasi-variational inequalities (QVIs) were first formulated and studied by Bensoussan and Lions \cite{BensoussanLionsArticle,LionsDistr} in the context of stochastic impulse control where solutions to the QVI of interest determine value functionals of the impulse control problem. In recent years QVIs have demonstrated to be versatile models for physical phenomena where nonsmoothness  and nonconvexity are prevalent. For example, this includes superconductivity \cite{KunzeRodrigues,BarrettPrigozhinSuperconductivity,Prigozhin,MR3335194,MR1765540}, sandpile formation and growth \cite{PrigozhinSandpile,BarrettPrigozhinSandpile,Prigozhin1996,Prigozhin1994,MR3335194}, and the determination of lakes and river networks \cite{Prigozhin1996,Prigozhin1994,MR3231973}, among other applications. 

In general, QVIs represent a step further in complexity in comparison to variational inequalities (VIs). The main difference resides in the fact that for QVIs the constraint set depends on the state variable itself instead of being constant as in the VI case. This adds to the nonsmooth and nonlinear nature of the problem and represents a major obstacle for the sensitivity study of this type of problem, but it also poses challenges on more fundamental levels: using arguments based on direct methods in the calculus of variations to show existence of solutions falls short of the task, and development of solution algorithms require a problem-tailored approach; see \cite{MR3648950,MR3119319,MR3023771,MR3335194,MR3231973,BarrettPrigozhinSandpile,BarrettPrigozhinSuperconductivity}. 

For VIs in infinite dimensions, the study of sensitivity and directional differentiability of the forcing term to state map and/or the associated metric projection operators has been considered in the fundamental works by Mignot \cite{Mignot}, Haraux \cite{Haraux1977} and Zarantonello \cite{Zarantonello1971}. In particular, the differentiability question is related to the polyhedricity property of the constraint set associated to the VI. Recent studies of sensitivity for VIs of the second kind can be found in \cite{HintThomas,Christof2018,ChristofWachsmuth}. Differentiability issues for the finite dimensional QVI case have been studied in \cite{MR2338444,MR1471270,MR1427449,MR1314522,kovcvara1995class}. To the best of our knowledge, the sensitivity analysis for QVIs in infinite dimensions is missing in the literature, and it is the purpose of the present work to bridge this gap.


We provide in what follows some basic assumptions on function spaces, operators, and notation, and a description of the structure of the main proof of the paper. 
Let $X$ be a locally compact topological space which is countable at infinity and let $\xi$ be a Radon measure on $X$. Suppose that there is a Hilbert space $V$ with $V \subset L^2(X;\xi)$ a continuous and dense embedding such that $|u| \in V$ and 
\[\text{there exists $\kappa > 0$ with $\lVert u^+\rVert_{V} \leq \kappa \norm{u}{V}$}\]
whenever $u \in V$. Let $A\colon V \to V^*$ be a linear operator satisfying the following properties for all $u \in V$: 
\begin{align*}
\langle Au, v \rangle &\leq C_b\norm{u}{V}\norm{v}{V}\tag{boundedness}\\
\langle Au, u \rangle &\geq C_a\norm{u}{V}^2\tag{coercivity}\\
\langle Au^+, u^- \rangle &\leq 0\tag{T-monotonicity}
\end{align*}
where $\langle \cdot, \cdot \rangle = \langle \cdot, \cdot \rangle_{V^*, V}$ is the standard duality pairing and we use the convention $u= u^+ - u^-$. 
Define the bilinear form $a\colon V \times V \to \mathbb{R}$ associated to the operator $A$ by $a(u,v) := \langle Au, v \rangle$. Under the above circumstances, $(V,a)$ falls into the class of \emph{positivity preserving coercive forms} with respect to $L^2(X;\xi)$ \cite{MaRockner,Bliedtner} (the `positivity preserving' terminology refers to the T-monotonicity but there are other equivalent conditions, see \cite[Proposition 1.3]{MaRockner}). We further assume that 
\begin{equation}\label{eq:assDensity}
V\cap C_c(X)  \subset C_c(X) \qquad \text{and}\qquad  V\cap C_c(X)  \subset V \qquad\text{are dense embeddings,}
\end{equation}
where the density is with respect to the supremum norm and the $V$ norm respectively. Forms $(V,a)$ that satisfy this density property are called \emph{regular} \cite[\S 1.1]{FukushimaBook} \cite[\S 2]{Bliedtner}.  This framework allows us to define the notions of capacity, quasi-continuity and related objects, see \cite[\S 3]{Mignot} and \cite[\S 3]{Haraux1977}.
We have in mind here $A$ as a linear elliptic differential operator and the space $V$ as a Sobolev space over a domain in $\mathbb{R}^n$. We will give more details of this and some concrete examples of the above definitions and spaces in \S \ref{sec:examples}.

\begin{remark}A space $V$ under all of the previous assumptions except the second density assumption in \eqref{eq:assDensity} is referred to by Mignot in \cite{Mignot} as a `Dirichlet space' --- this is rather inconsistent with the modern literature \cite{FukushimaBook} where Dirichlet spaces and Dirichlet forms are defined differently (see \cite[\S 1.1]{FukushimaBook}), for example in place of the T-monotonicity property the following \emph{Markov property} should hold:
\[\text{if $u \in V$ then $\bar u := \min(u^+,1) \in V$ and $a(\bar u,\bar u) \leq a(u,u)$}.\]
However, note that if the Markov property holds then so does T-monotonicity \cite[Remark 1.4]{MaRockner} \cite[Proposition 5]{Ancona} and hence a Dirichlet form is also a positivity preserving form.
\end{remark}
\begin{remark}
It should be possible to generalise the above setting of positivity preserving spaces $(V,a)$ to reflexive Banach spaces using for example the theory in \cite{MR2869508} which generalises \cite{Mignot}.
\end{remark}
Set $H:= L^2(X;\xi)$ and let $\Phi \colon H \to V$ be a possibly nonlinear map with $\Phi(0) \geq 0 \text{ a.e., }$ and suppose that it is increasing in the sense that $u \geq v$ a.e. (for $u, v \in H$) implies $\Phi(u) \geq \Phi(v)$ a.e. We define the set-valued mapping $\mathbf{K} \colon H \rightrightarrows V$ by
\begin{equation*}\label{eq:defnK}
\mathbf{K}(\varphi) := \{ v \in V : v \leq \Phi(\varphi)\text{ a.e.}\}.
\end{equation*}
For fixed $\varphi \in H$, $\mathbf{K}(\varphi)$ is a closed, convex and non-empty set. Given $f \in V^*$, consider the QVI
\begin{equation}\label{eq:QVI}
u \in \mathbf{K}(u) : \quad \langle Au-f, u-v \rangle \leq 0 \quad \forall v \in \mathbf{K}(u).\tag{P\textsubscript{QVI}}
\end{equation}
In general, there are multiple solutions of \eqref{eq:QVI} (this and the existence theory for the QVI will be discussed later on) so we denote by $\mathbf{Q}\colon V^* \rightrightarrows V$ the set-valued mapping that takes a source term into the set of solutions of \eqref{eq:QVI} with that right hand side source; hence \eqref{eq:QVI} reads $u \in \mathbf{Q}(f)$. When $\mathbf{K}$ is a constant mapping $\mathbf{K}(\varphi) \equiv \mathbf{K}$ the problem \eqref{eq:QVI} reduces to a standard variational inequality. In this work, we are interested in the differential sensitivity analysis of the map $$f \mapsto \mathbf{Q}(f);$$ more precisely, we wish to show that a particular realisation of the multi-valued map $\mathbf{Q}$ is directionally differentiable. Such a result is of independent interest in itself but it is also a necessary step for deriving first order stationarity conditions for optimal control problems where the state is related to the control through a QVI (the optimal control problem for VIs has been studied in the principal works \cite{Mignot,MR739836}). Furthermore, if the directional derivative can be suitably characterised then this information can be used to develop efficient bundle-free numerical solvers and solution algorithms as done in the case of the optimal control of the obstacle problem \cite{MHintermueller_TMSurowiec_MAPR}.


Let us discuss our approach and the difficulties encountered in the paper. The idea is to approximate $q(t) \in \mathbf{Q}(f+td)$ by a sequence of variational inequalities (each of which has a fixed obstacle), obtain suitable differential formulae for those VIs and then pass to the limit. There are some delicacies in this procedure:
\begin{itemize}
\item the derivation of the expansion formulae for the above-mentioned VI iterates takes some work, since they must relate $q(t)$ to a solution $u \in \mathbf{Q}(f),$ and recursion plays a highly nonlinear role in the relationship between one iterate and the preceding iterates
\item obtaining uniform bounds on the directional derivatives is not easy in the general case even though the derivatives satisfy a VI; it requires us to handle a recurrence inequality unless some regularity is available
\item proving that the higher order terms in the expansion formulae for the VI iterates converge in the limit to a term which is also higher order is difficult since this involves two limits and commutation of limits in general requires an additional uniform convergence.
\end{itemize}
Indeed, the main difficulty is the final point above. Although we do obtain some monotonicity properties of the directional derivatives and the higher order terms of the iterates, this information unfortunately does not help us as much as expected so more graft is needed to achieve our results. We will comment on this and the other technical difficulties throughout the paper as appropriate.

\subsection{Some definitions and assumptions on the data}
We define the closed convex cones
\begin{align*}
H_+ &:= L^2_{+}(X;\xi) := \{ v \in L^2(X;\xi) : v \geq 0 \text{ a.e.}\}\\
V_+ &:= \{ v \in V : v \geq 0 \text{ a.e.}\}.
\end{align*}
The latter can be used to give a canonical ordering to the dual space $V^*$ through the cone
\[V^*_+ := \{g \in V^* : \langle g, v \rangle_{V^*,V} \geq 0 \quad \forall v \in V_+\}\]
so that for elements $g, h \in V^*$, the inequality $g \geq h$ in $V^*$ is defined to mean $g-h \in V^*_+$. We also use the notation 
\[L^\infty_+(X;\xi) := \{v \in L^\infty(X;\xi) : v \geq 0\text{ a.e.}\}.\]

We first assume that the data $f \in V^*_+$ 
and define $\bar u \in V$ as the (non-negative) weak solution of the unconstrained problem
\begin{equation}\label{eq:baru}
\begin{aligned}
A\bar u &= f
\end{aligned}
\end{equation}
which is a linear PDE (and indeed it has a unique solution thanks to the Lax--Milgram lemma).

\begin{definition}[The function $u$]
We fix $u \in \mathbf{Q}(f)$ as an arbitrary solution of \eqref{eq:QVI} such that $u \in [0,\bar u]$\footnote{This exists by Theorem \ref{thm:existenceForQVI}.} (this means $0 \leq u \leq \bar u$ a.e. in $\Omega$).
\end{definition}
We introduce the following possible hypotheses on $\Phi$ (in addition to the ones we stated at the start, \textbf{which are always assumed to stand}). Note that we do not enforce all of these assumptions in every lemma or theorem; on the contrary we shall be selective so that we keep results applicable in as general a setting as possible. Below, we use the notation $B_R(z) := \{ v \in V : \norm{v-z}{V} \leq R\}$ to mean the closed ball of radius $R$ around $z \in V$.
\begin{enumerate}[label=(\textbf{A\arabic*})]
\item \label{itm:PhiDiff} 
The map $\Phi\colon V \to V$ is Hadamard directionally differentiable. That is, for all $v$ and all $h$ in $V$, the limit
\[\lim_{\substack{h' \to h\\t \to 0^+}}\frac{\Phi(v+t h')-\Phi(v)}{t}\]
exists in $V$, and we write the limit as $\Phi'(v)(h)$. Hence, if $h(t) \to h$, then
\begin{equation}\label{eq:derivativeOfPhi}
\Phi(v+th(t)) = \Phi(v) + t\Phi'(v)(h) + \hat l(t,h,h(t),v)
\end{equation}
holds where $\hat l$ is a higher order term, i.e.,  $t^{-1}\hat l(t,h,h(t),v) \to 0$ as $t \to 0^+$. We write $\hat l(t,h,h,v) = l(t,h,v)$  when $h(t) \equiv h$.





\item  We need one of the following:
\begin{enumerate}
\item\label{itm:compContOfPhi} $\Phi\colon V \to V$ is completely continuous, 
or
\item\label{itm:PhiConcaveEtc}  $V=H^1(\Omega)$, $X = \overline{\Omega}$ where $\Omega$ is a bounded Lipschitz domain, $\Phi\colon V \cap L^\infty_+(\Omega) \to L^\infty_+(\Omega)$ and is concave with $\Phi(0) \geq c > 0.$ 
\end{enumerate}
\item\label{itm:compContOfDerivOfPhi} 
The map $\Phi'(v)\colon V \to V$ is completely continuous (for fixed $v \in V$).
\item\label{item:assPhiDerivativeBounded} For any $b \in V$, $h\colon(0,T) \to V$ and $\lambda \in [0,1]$, 
\[\frac{\norm{\Phi'(u+tb+\lambda h(t))h(t)}{V}}{t} \to 0\quad\text{as $t \to 0^+$ whenever $\frac{h(t)}{t} \to 0$ as $t \to 0^+$}.\]


\item\label{itm:smallnessOfDerivOfPhi}
There exists $\epsilon > 0$ such that
\[\norm{\Phi'(z)(v)}{V} \leq C_\Phi\norm{v}{V} \quad \forall z \in B_\epsilon(u),\; \forall v \in V\]
with \[C_\Phi < \frac{1}{1+{C_a}^{-1}C_b},\]
%
where $C_a$ and $C_b$ are the constants of coercivity and boundedness from earlier. (This is a sufficient condition; what we really need is \eqref{eq:requiredSmallness} and there may be a better way of phrasing the assumption).
\end{enumerate}
Since $\Phi$ is Hadamard directionally differentiable, it is also compactly differentiable (see \cite{Shapiro}). This means that
\begin{equation*}\label{eq:lotOfL}
\frac{l(t,h,v)}{t} \to 0\quad\text{uniformly in $h$ on the compact subsets of $V$}.
\end{equation*}
Observe carefully that \ref{item:assPhiDerivativeBounded} and \ref{itm:smallnessOfDerivOfPhi} depend on the specific function $u$, i.e., these are \emph{local} conditions. Note also that the radius $\epsilon$ in \ref{itm:smallnessOfDerivOfPhi} can be arbitarily small, i.e., the stated boundedness of the derivative of $\Phi$ is required only around a closed ball at $u$ with arbitrarily small radius.
\begin{remark}In fact, \ref{itm:PhiDiff} can be weakened significantly by requiring Hadamard differentiability of $\Phi$ only around the point $u$, i.e., locally, as in assumptions \ref{item:assPhiDerivativeBounded} and \ref{itm:smallnessOfDerivOfPhi}.
\end{remark}
\begin{remark}[Compactness vs. complete continuity]
Recall that a compact operator maps bounded sets in the domain into sets with compact closure in the range. If $\Phi$ is completely continuous, then it is compact. But since $\Phi$ is allowed to be nonlinear, if $\Phi$ is compact, it does not necessarily follow that $\Phi$ is completely continuous. Hence the set of compact operators is larger than the set of completely continuous operators.
\end{remark}
\begin{remark}\label{rem:onPhi}Some comments regarding the assumptions on $\Phi$ are in order.
\begin{enumerate}
\item If $\Phi$ is directionally differentiable and Lipschitz, it is Hadamard differentiable.

\item If $\Phi$ is Lipschitz, the following relationship between the Lipschitz constant and $\Phi'$ holds:
\[\norm{\Phi'(v)h}{V} \leq \Lip{\Phi}\norm{h}{V}\quad\forall v,h \in V.\]
\item If $\Phi$ is a superposition operator (i.e., $\Phi(z)(x) = \Phi(z(x))$) and $\Phi \colon H^1(\Omega) \to H^1(\Omega)$ is compact, then $\Phi$ can only be a constant function. 

\item If \ref{itm:PhiConcaveEtc} holds and $f \in L^\infty_+(\Omega)$, then solutions of the QVI \eqref{eq:QVI} are unique \cite{Laetsch}.

\item Assumption \ref{itm:compContOfDerivOfPhi} need not imply the completely continuity of $\Phi$ itself (cf. \ref{itm:compContOfPhi}) however in the linear or affine case, $\Phi'(v)(h) = \Phi(h)$ and $l \equiv 0$ and 
\ref{itm:compContOfDerivOfPhi} is equivalent to \ref{itm:compContOfPhi}.
\item If $\Phi$ is linear, then $\hat l(t,h,h(t),v) = t\Phi(h(t)) - t\Phi(h)$ so if $h(t) = h + tb$, the remainder term is $\hat l(t,h,h(t),v) = t^2\Phi(b).$
\item Note that \ref{itm:smallnessOfDerivOfPhi} implies the existence of a constant $c > 0$ such that
\begin{equation}\label{eq:consequenceOfA5}
\norm{\Phi'(u)b}{V} 
\leq \frac{C_a-c}{C_b}\norm{b}{V},
\end{equation}
which we will use later on.
\item The assumption \ref{itm:smallnessOfDerivOfPhi}, in the case that $\Phi$ is linear, imposes a smallness condition on the operator norm of $\Phi$ which enforces uniqueness of solutions of the QVI. However, it does not necessarily rule out the multivalued setting in the case of nonlinear $\Phi$.
\end{enumerate}
\end{remark}
To state the main results of the paper, we first need some definitions. In a similar fashion to $\bar u$, define $\bar q(t) \in V$ as the solution of the unconstrained problem with right hand side $f+td$:
\begin{equation}\label{eq:barq}
\begin{aligned}
A\bar q(t) &= f+td.
\end{aligned}
\end{equation}
Define the zero level set mapping $\mathcal{Z}\colon V \rightrightarrows X $ and the coincidence set $\mathcal{A}\colon V \rightrightarrows X$ by
\begin{align*}
\mathcal{Z}(v) &:= \{ x \in X : v(x) = 0\}\\
\mathcal{A}(u) &:=\{x \in X : u(x) = \Phi(u)(x)\}
\end{align*}
 and observe that  
\begin{equation}\label{eq:AandZ}
\mathcal{A}(u) = \mathcal{Z}(u-\Phi(u)).
\end{equation}

\subsection{Examples}\label{sec:examples}
Having introduced the abstract positivity preserving forms $(V,a)$ earlier, we give now some concrete prototypical examples on domains in $\mathbb{R}^n$ with $\xi$ the Lebesgue measure. All of the following examples give rise to regular positivity preserving coercive forms in $L^2(X)$ (for various choices of $X$) satisfying the assumptions in the introduction.
\begin{enumerate}
\item Let $\Omega$ be a bounded Lipschitz domain, $V=H^1_0(\Omega)$ or $H^1(\Omega)$ and let $A$ be the linear second-order elliptic operator 
\[\langle Au, v \rangle =  \sum_{i,j=1}^n \int_\Omega a_{ij}\frac{\partial u}{\partial x_i}\frac{\partial v}{\partial x_j} + \sum_{i=1}^n \int_\Omega b_{i}\frac{\partial u}{\partial x_i}v + \int_\Omega c_0 uv\]
with coefficients $a_{ij}, b_i, c_0 \in L^\infty(\Omega)$ such that for all $\xi \in \mathbb{R}^n$ and for some $C>0$,
\[\sum_{i,j=1}^n a_{ij}\xi_i\xi_j \geq C|\xi|^2\quad\text{a.e., }\]
and $c_0 \geq \lambda > 0$ with $\lambda$ a constant. The space $X$ is
\begin{equation*}
X := \begin{cases}
\Omega : \text{if $V=H^1_0(\Omega)$}\\
\overline{\Omega} : \text{if $V=H^1(\Omega)$}.
\end{cases}
\end{equation*}
The choice of $\overline{\Omega}$ above ensures that the density condition \eqref{eq:assDensity} is fulfilled \cite[Example 1.6.1]{FukushimaBook}. The model example is $A=-\Delta + I$ (i.e., $a_{ij} = \delta_{ij}$, $b_i \equiv  0$ and $c_0 \equiv 1$), the Laplacian with a lower order term:
\begin{equation}\label{eq:modelCaseForA}
\langle Au, v \rangle = \int_\Omega \nabla u \cdot \nabla v + uv.
\end{equation}
\item Let $\Omega$ be the half space of $\mathbb{R}^d$ for $d \geq 2$,  $A$ be defined by \eqref{eq:modelCaseForA} with $V=H^1(\Omega)$ and $X=\overline{\Omega}$. This leads to a regular Dirichlet form \cite[\S1, Examples 1.5.3 and 1.6.2]{FukushimaBook}. The same is true for $\Omega=X=\mathbb{R}^d$ for any $d \geq 1$.
\item
Let $V=H^s(\Omega)$ for $s \in (0,1)$ on a bounded Lipschitz domain $\Omega$, where the classical fractional Sobolev space $H^s(\Omega)$ is defined as the subspace of $L^2(\Omega)$ with the following norm finite:
\begin{equation}\label{eq:normHs}
\norm{u}{H^s(\Omega)} := \left(\int_{\Omega} u^2 + \int_{\Omega}\int_{\Omega}\frac{|u(x)-u(y)|^2}{|x-y|^{n+2s}}\right)^{\frac 12}.
\end{equation}
Set $\langle Au, v \rangle = (u,v)_{H^s(\Omega)}
$. In this case, $X := \overline{\Omega}.$
\item Recall the singular integral definition of the fractional Laplacian for sufficiently smooth functions $u\colon \mathbb{R}^d \to \mathbb{R}$:
\[(-\Delta)^s u(x) := c\int_{\mathbb{R}^d}\frac{u(x)-u(y)}{|x-y|^{n+2s}}\;\mathrm{d}y, \quad\text{where}\quad c=\frac {4^{s}\Gamma (d/2+s)}{\pi ^{d/2}|\Gamma (-s)|},\]
again for $s \in (0,1)$. Pick $V=H^{s}(\mathbb{R}^d)$ (this space is defined through the norm \eqref{eq:normHs} but with $\Omega$ replaced with $\mathbb{R}^d$) and 
define the operator
\[\langle Au, v \rangle := \int_\Omega (-\Delta)^{s\slash 2}u (-\Delta)^{s\slash 2}v + \int_\Omega uv,\]
and here we choose $X = \mathbb{R}^d.$ Then $(V,a)$ is a regular Dirichlet form on $L^2(\mathbb{R}^d)$. 
\end{enumerate}
For full details of fractional Sobolev spaces and fractional Laplace operators, see for example \cite{MR0290095,MR2944369}. The first and third examples above are Examples 1 to 3 in \cite[\S 3]{Mignot}. 

As for $\Phi$, suppose we are in the functional setting of the first example above. We have in mind
\[\Phi(u) := L^{-1}u\]
where $L\colon V \to V^*$ is an appropriate second-order linear elliptic operator. If $\Omega$ is sufficiently smooth, elliptic regularity would yield $\Phi(u) \in V \cap H^{2}(\Omega)$  (eg. in case $L=-\Delta$ provided $\Omega$ is a $C^2$-domain). Validity of  a weak comparison principle would imply that $\Phi$ is increasing (for example if $L=-\Delta$, with $\Phi(f_i) =  L^{-1}f_i =: u_i$ and $f_1 \leq f_2$, rearrange to get $L(u_1-u_2)=f_1-f_2$ and then test with $(u_1-u_2)^+$). If a continuous dependence estimate of the form
\[\norm{L^{-1}f}{H^{1+\epsilon}(\Omega)} \leq C\norm{f}{L^2(\Omega)}\]
is available, then it would imply, along the linearity of $L$, that $\Phi$ is completely continuous from $V$ into $V$ due to the compact embedding $H^1(\Omega) \subset L^2(\Omega)$. Also $\Phi\colon H \to V$ is clearly Lipschitz and so is Hadamard differentiable as explained in Remark \ref{rem:onPhi}. Linearity also implies that the derivative is completely continuous with respect to the direction. In summary, assumptions \ref{itm:PhiDiff}, \ref{itm:compContOfPhi}, \ref{itm:compContOfDerivOfPhi} and \ref{item:assPhiDerivativeBounded} 
are satisfied.

\subsubsection{Application to fluid flow}Let us consider on the domain $D:=\mathbb{R}^{n-1}\times \mathbb{R}^+$ the pressure $U\colon D \to \mathbb{R}$ of an incompressible fluid on $D$. We think of $\Omega:=\partial D =\mathbb{R}^{n-1}\times \{0\}$ as a membrane which allows fluid to leave the domain $D$ but does not admit fluid into $D$. Let $\Psi$ represent an external pressure applied on the membrane $\Omega$. When the pressure $U=\Psi$ flow is admitted and $\partial_\nu U > 0$ holds.
Otherwise when the external pressure $\Psi$ exceeds $U$, i.e, $U < \Psi$, then there is no flow and $\partial_\nu U = 0$. Assume also that the compartment $D$ is connected to $\mathbb{R}^{n-1}\times \mathbb{R}^-$ via some mechanism such that if $U$ increases then $\Psi$ increases too; this describes the physically reasonable situation where the external pressure changes as the fluid $U$ enters $\mathbb{R}^{n-1}\times\mathbb{R}^-$. Thus $\Psi=\Psi(U)$ depends on $U$ and if we posit that $U$ satisfies in equilibrium 
\[-\Delta U = F\quad\text{on $\Omega$}\]
for a forcing term $F$ active on the membrane, then the following inequality is satisfied:
\begin{align*}
U \leq \Psi(U): \quad\int_{D} \nabla U \cdot \nabla (U-V) &\leq \int_\Omega f(U-V)\qquad \forall V : V \leq \Psi(U)
\end{align*}
where $f=F|_{\Omega}$ is a forcing term on the membrane. To recast this inequality in a suitable form, first recall the well-known fact that the half-Laplacian $(-\Delta)^{1\slash 2}$ of a function $w$ defined on $\mathbb{R}^{n-1}$ can be characterised as the Dirichlet-to-Neumann mapping (for example see \cite{CabreTan}) of the harmonic extension of $w$ onto $\mathbb{R}^{n-1} \times \mathbb{R}^+ = D$:
\begin{align*}
(-\Delta)^{1\slash 2}w := \partial_n W|_{\Omega} \quad \text{where} \quad
\begin{cases}
-\Delta W &=0 \quad\text{on $D$}\\
W &= w\quad\text{on $\Omega$}.
\end{cases}
\end{align*}
If we then define $u:=U|_\Omega$ and use the above (after recalling the definition of the weak normal derivative) we find that $u \in H^{1\slash 2}(\mathbb{R}^{n-1})$ satisfies
\begin{align*}
u \leq \Psi(u):
\quad \langle (-\Delta)^{1\slash 2}u, u-v \rangle &\leq \langle f, u-v \rangle\qquad \forall v \in H^{1\slash 2}(\mathbb{R}^{n-1}) : v \leq \Psi(u).
\end{align*}
If we add a regularisation term $\epsilon(u,u-v)$ to the left hand side of the inequality above, this application fits into our framework under appropriate assumptions on $\Psi$. This example is related to the Signorini problem, see \cite[Chapter 3]{Mosco1976} and \cite{2017arXiv171207001A}.

\subsection{Main results}
We now state the main theorems in this work, that of the directional differentiability for QVIs and a regularity result for the derivative under certain circumstances. Here and throughout the paper, we shall use the terminology q.e. to mean quasi-everywhere; a statement holds quasi-everywhere if it holds everywhere except on a set of capacity zero. For the definition of capacity and related notions we refer the reader to the texts \cite{Bonnans,Delfour}. 

We begin with the main sensitivity result.
\begin{theorem}\label{thm:main}
Given $f \in V^*_+$ and $d \in V^*_+$, for every $u \in \mathbf{Q}(f)\cap [0,\bar u]$, under assumptions \ref{itm:PhiDiff}, either \ref{itm:compContOfPhi} or both \ref{itm:PhiConcaveEtc} and $f, d \in L^\infty_+(\Omega)$, \ref{itm:compContOfDerivOfPhi},  \ref{item:assPhiDerivativeBounded} and  \ref{itm:smallnessOfDerivOfPhi}, there exists a function $q(t) \in \mathbf{Q}(f+td)\cap [u, \bar q(t)]$  and a function $\alpha=\alpha(d)\in V_+$ such that
\[q(t) = u + t\alpha + o(t)\quad\forall t > 0\]
holds where $t^{-1}o(t) \to 0$ as $t \to 0^+$ in $V$ and $\alpha$ satisfies the QVI
\begin{equation*}\label{eq:QVIforAlpha}
\begin{aligned}
&\alpha \in \mathcal{K}^u(\alpha) : \langle A\alpha - d, \alpha - v \rangle \leq 0 \quad \forall v \in \mathcal{K}^u(\alpha)\\
&\mathcal{K}^u(w) := \{ \varphi \in V : \varphi \leq \Phi'(u)(w) \text{ q.e. on }\mathcal{A}(u) \text{ and } \langle Au-f, \varphi -\Phi'(u)(w)\rangle = 0\}.
\end{aligned}
\end{equation*}
The directional derivative $\alpha=\alpha(d)$ is positively homogeneous in $d$. 
\end{theorem}
It is worth noting that under some weaker assumptions than those in Theorem \ref{thm:main} we can show that the \emph{expected} directional derivative of the QVI problem can be approximated by directional derivatives of VI iterates, see Theorem \ref{thm:boundednessOfDirectionalDerivatives} for this. As in \cite[Theorem 3.4]{Mignot}, we can in certain circumstances obtain a regularity result on the smoothness of the directional derivative found in Theorem \ref{thm:main}. We say that \emph{strict complementarity} holds if the set $\mathcal{K}^u$ simplifies to
\begin{equation}\label{eq:strictCompBeforeMainTheorem}
\mathcal K^u(w)=\mathcal S^u(w) := \{ \varphi \in V : \varphi = \Phi'(u)(w) \text{ q.e. on } \mathcal{A}(u) \}.
\end{equation}
If $\Phi'(u) \equiv 0$ (which is the case for VIs), then this condition simply asks for the reduction of the set $\mathcal{K}^u$ to a linear subspace. We discuss complementarity and strict complementarity in more detail in the next section.
Let us also introduce the following extra hypothesis.
\begin{enumerate}[label=(\textbf{B\arabic*})]
\item\label{itm:PhiDiffLinearInDirection} The map $h \mapsto \Phi'(v)(h)$ is linear for each $v$.
\end{enumerate}
\begin{theorem}\label{thm:regularity}
In the context of Theorem \ref{thm:main}, if strict complementarity holds, 
then the derivative $\alpha$  satisfies 
\begin{equation*}\label{eq:QVEforAlpha}
\begin{aligned}
\alpha \in \mathcal{S}^u(\alpha) &: \langle A\alpha - d, \alpha - v \rangle = 0 \quad \forall v \in \mathcal S^u(\alpha).
\end{aligned}
\end{equation*}
In this case, if \ref{itm:PhiDiffLinearInDirection} also holds, $\alpha=\alpha(d)$ satisfies $\alpha(c_1d_1 + c_2d_2) = c_1\alpha(d_1) + c_2\alpha(d_2)$ for positive constants $c_1$ and $c_2$ and non-negative directions $d_1$ and $d_2$ belonging to $V^*$. 
\end{theorem}

\section{Directional derivative formula for variations in the obstacle}

We first discuss variational inequalities and the directional derivatives associated to their solution mappings. Given data $f \in V^*$ and obstacle $\psi \in V$, consider the VI
\begin{equation}\label{eq:firstVI}
y \in \mathbf{K}(\psi): \quad \langle Ay-f, y -v \rangle \leq 0\quad\forall v \in \mathbf{K}(\psi).
\end{equation}
Define its solution mapping $S\colon V^* \times V \to V$ by $S(f,\psi)=y$; this is indeed well defined due to the Lions--Stampacchia theorem, see \cite[\S4.3]{Rodrigues} or \cite{Kinderlehrer} for example.  Since we are working with QVIs which by definition involve \emph{a priori} unknown obstacles, it becomes useful to be able to relate the problem \eqref{eq:firstVI} to a VI problem with zero obstacle but an extra source term. To achieve this, we make the change of variables 
\[\hat y := \Phi(\psi)-y,\quad\text{}\quad \hat v := \Phi(\psi) - v\]
(which is why the range of $\Phi$ must be in $V$) and reformulate \eqref{eq:firstVI} on the set $\mathbf{K}_0 := \{ v \in V : v \geq 0 \text{ a.e}\}$ as follows:
\begin{align}
\hat y \in\; &\mathbf{K}_0: \quad \langle A\hat y + f- A\Phi(\psi), \hat y - \hat v \rangle \leq 0\quad\forall v \in \mathbf{K}_0.\label{eq:1}
\end{align}
As a matter of notation, denote by $S_0\colon V^* \to V$ the solution mapping with $z=S_0(g)$ the solution of the following VI:
\begin{equation}\label{eq:S0VI}
z \in \mathbf{K}_0 : \quad \langle Az - g, z-v \rangle \leq 0\quad \forall v \in \mathbf{K}_0.
\end{equation}
Hence the solution of \eqref{eq:1} is $\hat y = S_0(A\Phi(\psi)-f)$, and by the definition of $\hat y$ we obtain the important formula relating $S$ and $S_0$:
\begin{equation}\label{eq:identitfySandS0}
S(f,\psi) = \Phi(\psi) - S_0(A\Phi(\psi)-f).
\end{equation}
Due to Mignot \cite{Mignot}, the mapping $S_0$ (and more generally solution mappings of VIs with non-zero but fixed obstacles) possesses a \emph{conical derivative} $S_0'(g)(d) \in V$ that satisfies
\begin{equation}\label{eq:mignot}
S_0(g+td) = S_0(g) + tS_0'(g)(d) + o(t,d,g)
\end{equation}
where the remainder term $o$ is such that $t^{-1}o(t,d,g) \to 0$ as $t \to 0^+$. The terminology \emph{conical} derivative refers to the directional derivative being positively homogeneous with respect to the direction and the associated limit that defines the derivative is taken along the positive half-line $t > 0$. 
This limit is uniform in $d$ on the compact subsets of $V^*$ (a fact of great utility later), and the derivative $\gamma := S_0'(g)(d)$ solves the VI \cite[Theorem 3.3]{Mignot}
\begin{equation}\label{eq:VIForS0}
\begin{aligned}
\gamma \in \mathcal{K}^z_0  &: \quad \langle A\gamma -d, \gamma - v \rangle \leq 0 \quad \forall v \in \mathcal{K}^z_0\\
\mathcal{K}^z_0 &:= \{ w \in V : w \geq 0 \text{ q.e. on } \mathcal{Z}(z) \text{ and } \langle Az-g, w \rangle = 0\},\quad z := S_0(g).
\end{aligned}
\end{equation}
Here the set $\mathcal{K}^z$ is well known in variational and convex analysis as the \emph{critical cone}. Since $S_0$ is Lipschitz, it is in fact Hadamard differentiable. This implies that if $d(t) \to d$ in $V^*$, then
\begin{equation}\label{eq:hadamardFormula}
S_0(g+td(t)) = S_0(g) + tS_0'(g)(d) + \hat o(t,d, d(t), g)
\end{equation}
holds with $t^{-1}\hat o(t,d,d(t),g) \to 0$ as $t \to 0^+$. 

Returning to the map $S$, we see that
\begin{align*}
S(f+td, \psi) -S(f,\psi) &= S_0(A\Phi(\psi)-f) - S_0(A\Phi(\psi)-f-td)\\
&= -tS_0'(A\Phi(\psi)-f)(-d) - o(t,-d,A\Phi(\psi)-f)
\end{align*}
so that
\begin{equation}\label{eq:relationS0AndSDerivatives}
\partial S(f,\psi)(d) = -S_0'(A\Phi(\psi)-f)(-d)
\end{equation}
is the directional derivative with respect to the source term of the solution mapping $S$ associated to the VI \eqref{eq:firstVI}. From \eqref{eq:VIForS0}, we see that $\delta:=\partial S(f,\psi)(d)$ satisfies
\begin{equation}\label{eq:VIForS}
\begin{aligned}
\delta \in \mathcal{K}^z  &: \quad \langle A\delta -d, \delta- v \rangle \leq 0 \quad \forall v \in \mathcal{K}^z\\
\mathcal{K}^z &:= \{ w \in V : w \leq 0 \text{ q.e. on } \mathcal{Z}(z-\Phi(\psi) ) \text{ and } \langle Az-f, w \rangle = 0\},\\
z &:= S(f, \psi).
\end{aligned}
\end{equation}
\begin{remark}An alternate characterisation of the critical cone in \eqref{eq:VIForS} in terms of q.e. statements was given by Wachsmuth in \cite[Lemmas 3.1 and A.5]{Wachsmuth}; namely the orthogonality condition in the critical cone above may be replaced by a quasi-everywhere equality on a subset of the active set which can be related to the notion of the fine support of $Az-f$. See \cite[Appendix A]{Wachsmuth} and also \cite{HarderWachsmuth} for more details.
\end{remark}
\begin{remark}
One may ask why we did not use \eqref{eq:identitfySandS0} to rewrite the QVI problem for $u$ directly in terms of $S_0$ and apply Mignot's theory without recourse to the iteration method that we will employ. Indeed, by \eqref{eq:identitfySandS0} we can write $u \in \mathbf{Q}(f)$ as $u=\Phi(u)-\hat u$ with $\hat u:= S_0(A\Phi(u)-f)$. Thus, since $u = (\Phi-I)^{-1}\hat u$, we have $\hat u = S_0(A\Phi(\Phi-I)^{-1}\hat u - f)$ which satisfies the VI
\[\hat u \in \mathbf{K}_0: \quad \langle A\hat u - A\Phi(\Phi-I)^{-1}\hat u + f, \hat u - \varphi \rangle \leq 0 \quad \forall \varphi \in \mathbf{K}_0.\]
Setting $\hat A:=A-A\Phi(\Phi-I)^{-1}$ and $\hat f =-f$, this reads
\[\hat u \in \mathbf{K}_0: \quad \langle \hat A\hat u -\hat f, \hat u - \varphi \rangle \leq 0 \quad \forall \varphi \in \mathbf{K}_0.\]
In general, the theory of Mignot cannot be applied to the solution mapping of this VI since $\hat A$ may not be linear (although see \cite{Levy} for a directional differentiability theory for nonlinear operators), nor coercive, nor T-monotone. However, let us consider a linear $\Phi$ that satisfies this property with $\norm{\Phi}{} < 1$ so that the Neumann series expansion is available. Then
\begin{align*}
\langle \hat Au, u \rangle &\geq C_a\norm{u}{V}^2 - \langle A\Phi(\Phi-I)^{-1}u, u \rangle\\
&\geq C_a\norm{u}{V}^2 -  C_b\norm{\Phi(\Phi-I)^{-1}u}{V}\norm{u}{V}\\
&\geq C_a\norm{u}{V}^2 -  C_b\norm{\Phi}{}\norm{(\Phi-I)^{-1}}{}\norm{u}{V}^2\\
&= C_a\norm{u}{V}^2 -  \frac{C_b\norm{\Phi}{}}{1-\norm{\Phi}{}}\norm{u}{V}^2\\
&=\frac{C_a-(C_a+C_b)\norm{\Phi}{}}{1-\norm{\Phi}{}}\norm{u}{V}^2
\end{align*}
so coercivity is achieved when $\Lip{\Phi} < C_a\slash (C_a+C_b)$ which puts us in the regime of unique solutions (due to \eqref{eq:ctsDependence} below), which is in agreement with our main theorem in this particular case.
\end{remark}
The VI \eqref{eq:S0VI} is equivalent to the following \emph{complementarity problem} \cite[\S 4.5, Proposition 5.6]{Rodrigues}:
\[z \in \mathbf{K}_0, \quad Az-g \geq 0 \text{ in $V^*$}, \quad \langle Az-g, z \rangle = 0.\]
Formally, complementarity refers to the idea that one or both of $Az-g$ and $z$ must vanish at any given point in $\Omega$ (i.e., they cannot simultaneously be strictly positive) since they are both non-negative and their (duality) product vanishes. Heuristically, we say that \emph{strict} complementarity holds if only one of these functions vanish at any given point, i.e., the set $\{Az-g=0\}\cap\{z=0\}$ (known as the \emph{biactive set}) is empty. For a discussion on how concepts such as the biactive set can be defined in the absence of sufficient regularity for $Az$ and $g$, see for example \cite{MR3307833,Gaevskaya2013}.

We, however, dispense with these notions and say that $S_0(g)=z$ satisfies \emph{strict complementarity} if the critical cone can be written as the linear subspace
\begin{equation}\label{eq:blah11}
\mathcal{K}^z_0 = \mathcal S^z_0 := \{ w \in V : w = 0 \text{ q.e. on } \mathcal{Z}(z)\}
\end{equation}
(this definition was used in \cite{Bonnans}). In this case, $S_0'(g)(d)$ is in fact a G\^{a}teaux derivative \cite[Theorem 3.4]{Mignot} \cite[\S 6.4, Corollary 6.60]{Bonnans}, i.e., it is linear with respect to the direction, and in lieu of \eqref{eq:VIForS0} it satisfies the weak formulation
\begin{equation*}\label{eq:PDEForS0}
\begin{aligned}
\gamma \in \mathcal S^z_0  &: \quad \langle A\gamma -d, \gamma - v \rangle = 0 \quad \forall v \in \mathcal S^z_0.
\end{aligned}
\end{equation*}
More generally (now for a VI with a non-trivial obstacle), we say that $S(f,\psi)=z$ satisfies \emph{strict complementarity} if the critical cone $\mathcal{K}^z$ can be written as
\[\mathcal{K}^z = \mathcal{S}^z := \{ w \in V : w = 0 \text{ q.e. on } \mathcal{Z}(z-\Phi(\psi) )\}.
\]
In Proposition \ref{prop:formulaForSDiff} we will introduce a further notion, that of strict complementarity with respect to a point in $V$, which can be thought of as a translation of the previous strict complementarity definition.

Checking that strict complementarity holds typically more requires regularity on the solution.

Returning to \eqref{eq:hadamardFormula}, we see by the next lemma that the convergence of the term $t^{-1}\hat o(t,d,d(t),g)$ is uniform with respect to $d$ on the compact subsets of $V^*$ in certain cases.
\begin{lemma}[Estimate on the higher order term of $S_0$]\label{lem:uniformConvergenceOfhatO}Let $b\colon (0,T) \to V^*$ satisfy 
 $t^{-1}b(t) \to 0$ as $t \to 0^+$. Then
\begin{equation*}
\frac{1}{t}\norm{\hat o(t,d, d + t^{-1}b(t),g)}{V} \leq \frac{1}{C_a}\frac{\norm{b(t)}{V^*}}{t} + \frac{\norm{o(t,d,g)}{V}}{t}\label{eq:oBound}
\end{equation*}
and thus $t^{-1}\hat o(t,d,d + t^{-1}b(t),g) \to 0$ in $V$ as $t \to 0^+$ (uniformly in $d$ on compact subsets provided $t^{-1}b(t) \to 0$ uniformly in $d$ on compact subsets).
\end{lemma}
\begin{proof}
Subtracting \eqref{eq:mignot} from \eqref{eq:hadamardFormula}, we obtain, since $S_0$ is Lipschitz,
\begin{align*}
\norm{\hat o(t,d, d + t^{-1}b(t),g) - o(t,d,g)}{V} &= \norm{S_0(g+t(d + t^{-1}b(t)))-S_0(g+td)}{V}\\
&\leq C_a^{-1}\norm{b(t)}{V^*}
\end{align*}
and this leads to the desired result after an application of the reverse triangle inequality.
\end{proof}
The next result records the higher order behaviour of the term $\hat l$ and is similar to Lemma \ref{lem:uniformConvergenceOfhatO} except we use a mean value theorem on Banach spaces \cite[\S2, Proposition 2.29]{Penot} on $\Phi$ instead of the Lipschitz property. 
\begin{lemma}[Estimate on the higher order term of $\Phi$]\label{lem:uniformConvergenceOfL}
Let \ref{itm:PhiDiff} hold and let $b\colon (0,T) \to V$ satisfy $t^{-1}b(t) \to 0$ as $t \to 0^+$. Then
\begin{equation}\label{eq:lBound}
\frac{1}{t}\norm{\hat l(t,h, h + t^{-1}b(t),v)}{V} \leq \sup_{\lambda \in [0,1]}\frac{\norm{\Phi'(v+th + \lambda b(t))b(t)}{V}}{t} + \frac{\norm{l(t,h,v)}{V}}{t}
\end{equation}
\end{lemma}
\begin{proof}
This follows from \eqref{eq:derivativeOfPhi} and the mean value theorem:
\begin{align*}
\norm{\hat l(t,h, h + t^{-1}b(t),v) - l(t,h,v)}{V} &= \norm{\Phi(v+t(h + t^{-1}b(t)))-\Phi(v+th)}{V}\\
&\leq \sup_{\lambda \in [0,1]}\norm{\Phi'(v+th + \lambda b(t))b(t)}{V}.
\end{align*}
\end{proof}

It is important to bear in mind that the higher order terms in the expansion formulae \eqref{eq:mignot} and \eqref{eq:derivativeOfPhi} for $S_0$ and $\Phi$ depend on the base points (respectively $g$ and $v$) too. See Remark \ref{rem:LOTsDependOnBasePoint} where we explain why this matters later on.

Two solutions $u_1=S(f_1,u_1)$ and $u_2=S(f_2,u_2)$ of the QVI with right hand sides $f_1$ and $f_2$ satisfy the following estimate (thanks to \eqref{eq:identitfySandS0})
\begin{align}
\nonumber \norm{u_1-u_2}{V} &\leq \norm{\Phi(u_1) - \Phi(u_2)}{V} + \norm{S_0(A\Phi(u_1)-f_1) - S_0(A\Phi(u_2)-f_2)}{V}\\
\nonumber &\leq \norm{\Phi(u_1) - \Phi(u_2)}{V} + C_a^{-1}\norm{A\Phi(u_1)-A\Phi(u_2) + f_2 - f_1}{V^*}\\
\nonumber &\leq (1+C_bC_a^{-1})\norm{\Phi(u_1) - \Phi(u_2)}{V} + C_a^{-1}\norm{f_1 - f_2}{V^*}\\
&\leq 
(1+C_bC_a^{-1})\sup_{\lambda \in (0,1)}\norm{\Phi'(\lambda u_1 + (1-\lambda)u_2)(u_1-u_2)}{V}+  C_a^{-1}\norm{f_1 - f_2}{V^*}
\label{eq:ctsDependence}
\end{align}
for some $\lambda \in (0,1)$ by the mean value theorem. If $\Phi$ is for example G\^{a}teaux differentiable then this implies
\begin{align*}
\norm{u_1-u_2}{V} &\leq (1+C_bC_a^{-1})\sup_{\lambda \in (0,1)}\norm{\Phi'(\lambda u_1 + (1-\lambda)u_2)}{op}\norm{u_1-u_2}{V}+  C_a^{-1}\norm{f_1 - f_2}{V^*}.
\end{align*}
The operator norm on the right hand side of the above can be replaced with the Lipschitz constant of $\Phi$ if $\Phi$ is Lipschitz. This shows that if the Lipschitz constant of $\Phi$ or the operator norm of $\Phi'$ is small enough, then solutions to the QVI are unique. Under these smallness conditions, the above estimate yields
\[\frac{\norm{\mathbf{Q}(f+td)-\mathbf{Q}(f)}{V}}{t} \leq C\norm{d}{V^*}\]
i.e., the difference quotient is bounded. Hence one may try to show the strong convergence in the limit $t \to 0$ of the difference quotient in order to show the existence of the directional derivative. This would require careful analysis of moving sets and cones and is a possible alternative approach to what we do here.

The next result gives a directional differentiability result for the solution mapping $S$ of the VI \eqref{eq:firstVI} with respect to the right hand side source term \emph{and also} the obstacle. It is clear then that information about the behaviour of $\Phi$ with respect to perturbations in the argument is needed for this. A related result can be found in the work of Dentcheva \cite{Dentcheva2001} in which the directional differentiability of metric projections onto moving convex subsets is studied under some assumptions on the differentiability of the constraint set mapping, in a general normed space setting. For stability and continuity results with respect to variations in the obstacle, see \cite{Mosco,MR653202,Toyoizumi1991,Rodrigues}.
\begin{proposition}\label{prop:formulaForSDiff}
Let \ref{itm:PhiDiff} hold. Let $f, d \in V^*$, $v,b \in V$ and $h\colon (0,T) \to V$. 

\noindent(1) For $t > 0$, the expansion formula 
\[S(f+td, v + tb + h(t)) = S(f,v) + tS'(f,v)(d,b) + r(t,b,h(t),v)\]
holds where 
\begin{align*}
&S'(f,v)(d,b):= \Phi'(v)(b) + \partial S(f,v)(d-A\Phi'(v)(b))\\
&r(t,b,h,v) := \hat l(t,b,b+t^{-1}h,v) -\hat o(t,A\Phi'(v)(b)-d, A\Phi'(v)(b)- d + At^{-1}l(t,b,b+t^{-1}h,v), A\Phi(v)-f )
\end{align*}
and $\alpha(d,b):=S'(f,v)(d,b)$ is positive homogeneous in the sense that $\alpha(kd,kb) = k\alpha(d,b)$ for any $k  > 0$, and it satisfies the VI
\begin{align*}
\alpha \in \mathcal{K}^z(b) &: \langle A\alpha -d, \alpha-\varphi \rangle \leq 0 \quad \forall \varphi \in \mathcal{K}^z(b)\\
\mathcal{K}^z(b) &:= \{ w \in V : w \leq \Phi'(v)(b) \text{ q.e. on } \mathcal{Z}(z-\Phi(v))\text{ and } \langle Az-f, w - \Phi'(v)(b) \rangle = 0\}\\
z&:= S(f,v).
\end{align*}
(2) If $t^{-1}h(t) \to 0$ as $t \to 0^+$ and \ref{item:assPhiDerivativeBounded} holds, the remainder term $r$ satisfies
\[\frac{r(t,b,h(t),v)}{t} \to 0 \quad\text{as $t \to 0^+$ uniformly in $d$ on the compact subsets of $V^*$}\]
and thus $S$ is conically differentiable with derivative $S'(f,v)(d,b).$ If \ref{item:assPhiDerivativeBounded} holds uniformly with respect to $b$ in the compact subsets of $V$, then the above stated convergence of $r$ is also uniform with respect to $b$ in the compact subsets of $V$.

\noindent (3) We say that $z=S(f,v)$ satisfies \emph{strict complementarity with respect to $b$} if
\[\mathcal K^z(b) = \mathcal S^z(b) := \{ w \in V : w = \Phi'(v)(b) \text{ q.e. on } \mathcal{Z}(z-\Phi(v))\}\]
holds, and in this case 
$\alpha$ satisfies
\begin{align*}
\alpha \in \mathcal S^z(b) &: \langle A\alpha -d, \alpha-\varphi \rangle =0 \quad \forall \varphi \in \mathcal S^z(b),
\end{align*}
and if \ref{itm:PhiDiffLinearInDirection} also holds, $(d,b) \mapsto \alpha(d,b)$ is linear and hence $\alpha=S'(f,v)(d,b)$ is a G\^{a}teaux derivative.
\end{proposition}
\begin{proof}
(1) The left hand side of the expansion formula to be proved is
\begin{equation}\label{eq:0}
S(f+td, v + tb + h(t)) = \Phi(v+tb+h(t)) - S_0(A(\Phi(v + tb + h(t)))-f-td)
\end{equation}
The first term on the right hand side can be written using the expansion formula \eqref{eq:derivativeOfPhi} for $\Phi$:
\begin{equation}
\Phi(v+tb+h(t)) = \Phi(v) + t\Phi'(v)(b) + \hat l(t,b, b + t^{-1}h(t), v)\label{eq:2},
\end{equation}
and using this and the expansion formula \eqref{eq:hadamardFormula} for $S_0$, we can write the second term on the right hand side as
\begin{align}
\nonumber S_0(A(\Phi(v + tb + h(t)))-f-td) &= S_0(A\Phi(v) - f + t(A\Phi'(v)(b) - d + t^{-1}A\hat l(t,b,b+t^{-1}h(t),v)))\\
\nonumber  &= S_0(A\Phi(v) - f) + tS_0'(A\Phi(v) - f)[A\Phi'(v)(b) - d]\\
&\quad  + \hat o (t,A\Phi'(v)(b) - d, A\Phi'(v)(b) - d + t^{-1}A\hat l(t,b,b+t^{-1}h(t),v), A\Phi(v)-f)\label{eq:3}
\end{align}
Now, plugging \eqref{eq:2} and \eqref{eq:3} into \eqref{eq:0} and using 
\[S_0'(A\Phi(v)-f)[A\Phi'(v)(b)-d]=-\partial S(f,v)(d-A\Phi'(v)(b))\] (which is the relation \eqref{eq:relationS0AndSDerivatives} between the directional derivatives of $S_0$ and $S$), we find that \eqref{eq:0} becomes
\begin{align*}
S(f+td, v + tb + h(t))
&= \Phi(v)+t\Phi'(v)(b)+\hat l(t,b,b+t^{-1}h(t),v) - S_0(A\Phi(v) - f)\\
&\quad  - tS_0'(A\Phi(v) - f)[A\Phi'(v)(b) - d]\\
&\quad-\hat o(t,A\Phi'(v)(b)-d, A\Phi'(v)(b)- d + t^{-1}A\hat l(t,b,b+t^{-1}h(t),v), A\Phi(v)-f)\\
&= S(f,v) +t(\Phi'(v)(b)+\partial S(f,v)(d-A\Phi'(v)(b)))+\hat l(t,b,b+t^{-1}h(t),v)\\
&\quad -\hat o(t,A\Phi'(v)(b)-d, A\Phi'(v)(b)- d +t^{-1}A\hat l(t,b,b+t^{-1}h(t),v), A\Phi(v)-f)
\end{align*}
which is what we needed to show. It is also clear that $\alpha:= \Phi'(v)(b)+\partial S(f,v)(d-A\Phi'(v)(b))$ is positively homogeneous. From \eqref{eq:VIForS}, the function $\delta :=\partial S(f,v)(d-A\Phi'(v)(b))$ satisfies 
\begin{align*}
\delta \in \mathcal{K}^z &: \quad \langle A\delta - d + A\Phi'(v)(b), \delta - \psi \rangle \leq 0 \quad \forall \psi \in \mathcal{K}^z\\
\mathcal{K}^z &:= \{ w \in V : w \leq 0 \text{ q.e. on } \mathcal{Z}(z-\Phi(v)) \text{ and } \langle Az-f, w \rangle = 0\},\\
\quad z &:= S(f,v).
\end{align*}
Recalling the definition of $\alpha$ and making the substitution $\varphi := \Phi'(v)(b)+ \psi$ in the above variational formulation for $\delta$ yields the formulation for $\alpha$ stated in the proposition.

\noindent (2) We estimate the remainder term by using Lemmas \ref{lem:uniformConvergenceOfhatO} and \ref{lem:uniformConvergenceOfL} as follows:
\begin{align*}
\norm{r(t,b,h(t),v)}{V}&\leq \norm{\hat l(t,b,b+t^{-1}h(t),v)}{V}\\
&\quad + \norm{\hat o(t,A\Phi'(v)(b)-d, A\Phi'(v)(b)- d + At^{-1}\hat l(t,b,b+t^{-1}h(t),v), A\Phi(v)-f )}{V}\\
&\leq \sup_{\lambda \in [0,1]}\norm{\Phi'(v + tb + \lambda h(t))h(t)}{V} + \norm{l(t,b,v)}{V} + \frac{C_b}{C_a}\norm{\hat l(t,b,b+t^{-1}h(t),v)}{H^{1}}\\
&\quad + \norm{o(t,A\Phi'(v)(b)-d, A\Phi(v)-f)}{V}\\
&\leq \sup_{\lambda \in [0,1]}\norm{\Phi'(v + tb + \lambda h(t))h(t)}{V}+ \norm{l(t,b,v)}{V}\\
&\quad + \frac{C_b}{C_a}\left(\sup_{\lambda \in [0,1]}\norm{\Phi'(v + tb + \lambda h(t))h(t)}{op}+ \norm{l(t,b,v)}{V}\right)\\
&\quad + \norm{o(t,A\Phi'(v)(b)-d, A\Phi(v)-f)}{V}\\
&\leq \left(1+\frac{C_b}{C_a}\right)\sup_{\lambda \in [0,1]}\norm{\Phi'(v + tb + \lambda h(t))h(t)}{V}+ \left(1+\frac{C_b}{C_a}\right)\norm{l(t,b,v)}{V} \\
&\quad + \norm{o(t,A\Phi'(v)(b)-d, A\Phi(v)-f)}{V}.
\end{align*}
Dividing by $t$ and sending $t \to 0^+$, we see that the remainder term vanishes in the limit thanks to \ref{item:assPhiDerivativeBounded}. Furthermore the convergence to zero is uniform in $d$ on compact subsets since $d$ appears only in the final term above. Since by \ref{itm:PhiDiff} $\Phi$ is compactly differentiable, the convergence to zero is also uniform in $b$ on compact subsets if also first term on the right hand side converges uniformly in $b$.

\noindent (3) If $S(f,v)$ satisfies strict complementarity (see \eqref{eq:blah11} and the surrounding discussion), 
that is, if
\[\mathcal K^z = \mathcal S^z := \{ w \in V : w = 0 \text{ q.e. on } \mathcal{Z}(\Phi(v)-z)\},\]
then $\delta$ satisfies
\begin{align*}
\delta \in \mathcal S^z &: \quad \langle A\delta - d + A\Phi'(v)(b), \delta - \psi \rangle = 0 \quad \forall \psi \in \mathcal S^z.
\end{align*}
As before, recalling the definition of $\alpha$ and making the substitution $\varphi := \Phi'(v)(b)+ \psi$ in the above equality yields the equality for for $\alpha$.
It follows under \ref{itm:PhiDiffLinearInDirection} that the mapping $\delta\colon V^* \times V \to V$ given by $(d,b) \mapsto \delta(d,b)$ is linear:
\[\delta(c_1 d_1 + c_2d_2, c_1b_1 + c_2b_2) = c_1\delta(d_1, b_1) + c_2\delta(d_2, b_2)\]
where $c_1, c_2 \in \mathbb{R}$, and $\alpha$ also inherits this property.
\end{proof}
\begin{remark}In the formulation of the previous proposition, we introduced the notion of a function satisfying strict complementarity with respect to a base point. This is compatible with how strict complementarity was defined in the paragraphs preceding Theorem \ref{thm:regularity}; indeed \eqref{eq:strictCompBeforeMainTheorem} can be rephrased in the language of the above proposition as $u=S(f,u)$ satisfies strict complementarity with respect to $w$ once we recall the identity \eqref{eq:AandZ}. 
\end{remark}

%
%

\section{Existence for the QVI and iteration scheme for sensitivity analysis}
In this section, we detail and set up the existence results that we need and also define the VI iteration scheme. We also specify the precise selection mechanism that will arise in the main directional differentiability result.
\subsection{Existence for QVIs in ordered intervals}
In this subsection we will justify the non-emptiness of the set $\mathbf{Q}(f)$. More precisely, recall that $u \in \mathbf{Q}(f)$ was chosen arbitrarily in the interval $[0,\bar u]$ with $\bar u \in V$ the weak solution of the unconstrained problem \eqref{eq:baru}. 
Let us see why such a $u$ actually exists.

We recall the notion of subsolution and supersolution for the mapping $S$ with right hand side $f$. A function $w$ is a \emph{subsolution} if $w \leq S(f,w)$, and a \emph{supersolution} is defined in the same way with the opposite inequality. If $f \geq 0$ in $V^*$ and since we assumed $\Phi(0) \geq 0$, it is not hard to check that $0$ is a subsolution for \eqref{eq:QVI}, 
and we also see that $\bar u$ is a supersolution by an argument like eg. \cite[Lemma 1.1, Chapter 4.1]{LionsBensoussan}. Standard results on elliptic PDEs guarantee that $\bar u \in V_+$ whenever $f \in V^*_+(\Omega)$.
\begin{theorem}\label{thm:existenceForQVI}
If $f \geq 0$ in $V^*$, there exist solutions $u \in V$ to \eqref{eq:QVI} in the interval $[0,\bar u].$ 
\end{theorem}
\begin{proof}
Tartar in \cite{Tartar74} proved, using the theory of fixed points in vector lattices in the work of Birkhoff \cite{Birkhoff}, that the subset of solutions $u$ of \eqref{eq:QVI} lying between the subsolution $0$ and supersolution $\bar u$ is non-empty (and in fact there exist smallest and largest solutions). See also Aubin \cite[Chapter 15.2.2]{Aubin} and Mosco \cite[Chapter 2.5]{Mosco1976}.
\end{proof}
This argument of course also applies to the QVI with non-negative right hand side $f+td$, resulting in existence of solutions on the interval $[0,\bar q(t)]$. But we want to localise  to a smaller subinterval $[u, \bar q(t)]$ and prove that solutions exist there. For this purpose, we have the following lemmas.
\begin{lemma}\label{lem:uIsSubSoln}If $f, d \geq 0$ in $V^*$, the function $u$ is a subsolution for the QVI with right hand side $f+td$.
\end{lemma}
\begin{proof}
The function $s=S(f+td,u)$ solves
\begin{align*}
s \in \mathbf{K}(u) : \quad \langle As - (f+td), s-v\rangle \leq 0\quad\forall v \in \mathbf{K}(u),
\end{align*}
wherein picking $v = s + (u-s)^+$ and also testing \eqref{eq:QVI}
 with $u - (u-s)^+$ and combining, we find
 \[\langle A(u-s), (u-s)^+\rangle \leq -\langle td, (u-s)^+ \rangle \leq 0.\]
Thanks to the T-monotonicity of $A$, this implies that $u \leq s$, i.e., $u \leq S(f+td,u)$. 
\end{proof} 
\begin{lemma}\label{lem:uIsBounded}If $f, d \geq 0$ in $V^*$, we have $u \leq \bar q(t)$ (where $\bar q(t)$ is defined in \eqref{eq:barq}).
\end{lemma}
\begin{proof}
Consider the difference of $A\bar u = f$ and $A\bar q(t) = f+td$:
\[A(\bar u - \bar q(t)) = -td.\]
Testing with $(\bar u-\bar q(t))^+$ and using $d \in V_+^*$ 
we find that $\bar u \leq \bar q$ a.e. Since by definition $u \leq \bar u$, we obtain the desired result.
\end{proof}
\begin{theorem}If $f, d \geq 0$ in $V^*$, there exist solutions $q_t \in V$  to 
\[q_t \in \mathbf{K}(q_t) : \quad \langle Aq_t-(f+td), q_t-v \rangle \leq 0 \quad \forall v \in \mathbf{K}(q_t)\]
in the interval $[u, \bar q(t)].$
\end{theorem}
\begin{proof}
Since $u$ is a subsolution (by Lemma \ref{lem:uIsSubSoln}) and $\bar q$ is a supersolution with $u \leq \bar q$ (by Lemma \ref{lem:uIsBounded}), we know again by \cite{Tartar74} that there exist solutions $\mathbf{Q}(f+td)$ now in the ordered interval $[u, \bar q(t)]$. 
\end{proof}
\subsection{Approximation of the QVI by VI iterates}
We now clarify the procedure laid out in the introduction as how to we tackle the problem. Define the sequence $q_n(t)$ by 
\begin{equation}\label{eq:sequenceQn}
q_n(t) \in \mathbf{K}(q_{n-1}(t)) : \quad \langle Aq_n(t) - (f+td), q_n(t) - v\rangle \leq 0 \quad \forall v \in \mathbf{K}(q_{n-1}(t))
\end{equation}
and $q_0$ is chosen so that $0 \leq q_0 \leq \bar q(t)$. We shall choose $q_0$ so that $q_n(t) \to \mathbf{Q}(f+td)$ in various senses (see later). Our idea is to obtain by Proposition \ref{prop:formulaForSDiff} an expression for $q_n(t)$ in terms of $u$, a derivative term $\alpha_n$ and a higher order term $o_n$, and then pass to the limit in this expression to hopefully obtain an expansion formula that will tell us that the QVI solution map is differentiable.


From now on we assume that the source term $f$ and the direction $d$ are non-negative in $V^*$ so that the results of the previous subsection are valid.  
\begin{remark}
The requirement for $f \in V^*$  to be non-negative was used to show that $0$ is a subsolution to \eqref{eq:QVI}. We could instead have assumed the existence of a subsolution and kept $f \in V^*$ more general, and similarly, we could also have chosen a different upper bound instead of $\bar u$ and $\bar q(t)$. However, for simplicity we will not proceed with this generalisation.
\end{remark}
\begin{theorem}\label{thm:convergenceOfIterations}Let $f, d \geq 0$ in $V^*$ and let either \ref{itm:compContOfPhi} or  \ref{itm:PhiConcaveEtc} and $f, d \in L^\infty_+(\Omega)$ hold. With $q_0 := u$, the sequence $q_n(t)=S(f+td,q_{n-1}(t))$ defined above is monotonically increasing in $n$ and converges to a function $q(t) \in \mathbf{Q}(f+td)$. The convergence is strong in $V$ if \ref{itm:compContOfPhi} holds, otherwise if \ref{itm:PhiConcaveEtc} and $f, d \in L^\infty_+(\Omega)$ hold, it is weak in $V$ and strong in $L^\infty(\Omega)$.

\end{theorem}
\begin{proof}
Since $q_1(t) = S(f+td, u)$ and $q_0=u=S(f,u)$ with $d \geq 0$ in $V^*$, it follows by the comparison principle \cite[\S 4.5, Corollary 5.2]{Rodrigues} that $q_1(t) \geq q_0.$ Similarly, since $q_2(t) = S(f+td, q_1(t))$ and $q_1(t) = S(f+td, q_0)$ and we have shown that $q_1(t) \geq q_0$, the comparison principle again (this time the comparison is in the obstacles, using the increasing property of $\Phi$) gives $q_2(t) \geq q_1(t)$. Repeating this argument, we find that $q_n(t)$ is an increasing sequence.

It is not hard to see that
\[q_n(t) \leq \bar q(t)\]
almost everywhere, by testing \eqref{eq:sequenceQn} with $q_n(t) - (q_n(t)-\bar q(t))^+$ and the $\bar q(t)$ PDE \eqref{eq:barq} with $(q_n(t)-\bar q(t))^+$ and combining in the usual way. Testing \eqref{eq:sequenceQn} with $q_n(t)\slash 2$ we obtain 
\[\norm{q_n(t)}{V} \leq C\]
where $C$ is independent of $n$ (but dependent on $f$ and $d$). Thus $q_{n_j}(t) \weaklyto q(t)$ in $V$ for a subsequence. The bounded increasing sequence $q_n(t)$ converges pointwise a.e. to a function $q(t)$ so by the subsequence principle we in fact have the following convergence for the whole sequence:
\begin{equation}\label{eq:convergencesForQn}
\begin{aligned}
q_n(t) &\weaklyto q(t) &&\text{in $V$}\\
q_n(t) &\to q(t) &&\text{a.e}.
\end{aligned}
\end{equation}
It remains for us to show that $q(t)$ solves the QVI problem.
\paragraph{In the case \ref{itm:compContOfPhi}.}The complete continuity of $\Phi$ tells us that $\Phi(q_{n-1}(t)) \to \Phi(q(t))$ strongly in $V$ and by the continuous embedding into $L^2(X)$, we have for a subsequence $\Phi(q_{n_j}(t)) \to \Phi(q(t))$ almost everywhere. Passing to the limit pointwise a.e. in $q_{n_j}(t) \leq \Phi(q_{{n_j}-1}(t))$ shows that $q(t)$ is feasible. The strong convergence $q_n(t) \to q(t)$ in $V$ follows from the standard continuous dependence estimate for VIs.


Given a function $v \in \mathbf{K}(q(t))$, setting $v_n:= v + \Phi(q_{n-1}(t))-\Phi(q(t))$, we see that  $v_n \leq \Phi(q_{n-1}(t))$ (so it is an admissible test function for the VI for $q_{n}(t)$) and 
$v_n \to v$ in $V$. This and the above convergence results plus the weak lower semicontinuity of norms is enough to the pass to the limit in \eqref{eq:sequenceQn} and we will obtain
\[\langle Aq(t)- (f+td), q(t)-v \rangle \leq 0\]
for all $v \in \mathbf{K}(q(t))$. This shows that $q(t) \in \mathbf{Q}(f+td)$. 

\paragraph{In the case \ref{itm:PhiConcaveEtc} and if $f, d \in L^\infty_+(\Omega)$.}As $\Phi(0) \geq c > 0$ the solution of \eqref{eq:QVI} is unique \cite{Laetsch}, and as $\Phi$ is concave and bounded away from zero at zero, it is known that $\mathbf{Q}(f+td)$ can be approximated by the iterations $q_n$ 
not only in the same sense as \eqref{eq:convergencesForQn} but also 
\begin{align*}
q_n(t) &\to q(t) \text{ in $L^\infty(\Omega)$}
\end{align*}
due to Hanouzet and Joly \cite{HanoJoly} (see also \cite[Appendix 1, \S 7]{Glowinski}), so long as 
$0 \leq q_0 \leq \bar q$ which is the case here by Lemma \ref{lem:uIsBounded}. Indeed, the concavity of $\Phi$ allows us to deduce that $q_n \to q$ directly, without going through the procedure described in the previous case where we had to approximate the test functions of the limiting QVI by test functions of the VI iterations, requiring compactness of $\Phi$.
\end{proof}

\begin{remark}
Lions and Bensoussan (see \cite[Chapter 4]{LionsBensoussan}) have shown that when $\Phi$ is of impulse control type in the $H^1_0(\Omega)$ or $H^1(\Omega)$ setting, if $u_n := S(f,u_{n-1})$, then
\begin{itemize}
\item setting $u_0 = \bar u$ leads to a decreasing sequence $u_n$ that converges to the maximal solution of \eqref{eq:QVI} in $[0,\bar u]$ \cite[Chapter 4, Lemma 1.2]{LionsBensoussan}
\item setting $u_0 = 0$ leads to an increasing sequence $u_n$ that may not converge to the minimal solution of the QVI; this is an open question.
\end{itemize}
\end{remark}
\subsubsection{Selection mechanism}To summarise the above, notice that the construction in the proof of Theorem \ref{thm:convergenceOfIterations} defines a selection mechanism for the multi-valued QVI solution mapping $\mathbf{Q}$ and this mechanism will appear in the expansion formula that characterises the directional derivative for the QVI. More precisely, we may choose any selection mapping $s_1\colon V^*_+ \to V$ that satisfies 
\[s_1(g) \in \mathbf{Q}(g)\cap [0,\bar u]. 
\]
Define the mapping $m_t \colon  \{h \in V^*_+ : h \leq f\text{ in $V^*$}\} \times \{w \in V : w \in [0,\bar q(t))\} \to V$ by
\begin{align*}
m_t(g,v) := \lim_{n \to \infty} v_n(t), \quad\text{where } 
\begin{cases}
v_0 &:= v,\\
v_n(t) &:= S(g+td, v_{n-1})
\end{cases}
\end{align*}
which satisfies $m_t(g,v) \in \mathbf{Q}(g+td)$. Then $q(t) = m_t(f,s_1(f))$ and $u=m_0(f,u)$.
\subsection{Expansion formula for the VI iterates}
Now that we know that the $q_n(t)$ converge to $q(t) \in \mathbf{Q}(f+td)$, we concentrate on deriving an expansion formula for $q_n(t)$ in terms of $u$.
Using Proposition \ref{prop:formulaForSDiff}, we can calculate the expansion formula for $q_1$ explicitly in terms of $u$:
\begin{equation}\label{eq:q1}
q_1 = S(f+td, q_0) = u + t\delta_1 + r(t,0,0,u),\qquad
\delta_1 
= \partial S(f,u)(d)
\end{equation}
(we could also have directly used \cite{Mignot} here, since there is no perturbation in the obstacle and Mignot's theory applies immediately).  Using this representation, we bootstrap and apply Proposition \ref{prop:formulaForSDiff} again to find $q_2$, and then $q_3$, explicitly in terms of $u$ and the directional derivatives of the previous step:
\begin{equation}\label{eq:systemforQn}
\begin{aligned}
q_2 &= S(f+td, u + t\delta_1 + r(t,0,0,u)) = u + t(\Phi'(u)\delta_1 + \delta_2) + r(t,\delta_1, r(t,0,0,u),u)\\
q_3 &= S(f+td, u + t(\Phi'(u)\delta_1+\delta_2) + r(t,\delta_1, r(t,0,0,u),u))\\
&= u + t(\Phi'(u)[\Phi'(u)(\delta_1) + \delta_2]+\delta_3) + r(t,\Phi'(u)(\delta_1)+\delta_2, r(t,\delta_1,r(t,0,0,u),u),u).
\end{aligned}
\end{equation}
where we have defined
\begin{alignat*}{3}
&\delta_2 
= \partial S(f,u)(d-A\Phi'(u)(\delta_1))\\
&\delta_3 
= \partial S(f,u)(d-A\Phi'(u)(\Phi'(u)(\delta_1)+\delta_2)).
\end{alignat*}
This inspires us to make the following definition for the general case:
\begin{align}\label{eq:defnDeltan}
\delta_n 
&:= \partial S(f,u)[d-A\Phi'(u)(\Phi'(u)[...\Phi'(u)[\Phi'(u)(\delta_0)+\delta_1] + \delta_2...] + \delta_{n-2}] + \delta_{n-1})]
\end{align}
To ease notation, define
\begin{equation}\label{eq:alphan}
\begin{aligned}
\alpha_n &:= 
\begin{cases}
\delta_1&:\text{if $n =1$}\\
\Phi'(u)[\Phi'(u)[...\Phi'(u)[\Phi'(u)(\delta_1)+\delta_2] + \delta_3...] + \delta_{n-1}] + \delta_{n} &:\text{if $n \geq 2$}
\end{cases}
\end{aligned}
\end{equation}
and observe the recursion formula
\begin{align}
\alpha_n =
\Phi'(u)[\alpha_{n-1}] + \delta_{n}\quad\text{for $n \geq 2$}\label{eq:idForAlphaN}
\end{align}
and the formula \eqref{eq:defnDeltan} defining $\delta_n$ can be written as
\begin{equation}\label{eq:deltaNandAlphaN}
\delta_n 
= \partial S(f,u)(d-A\Phi'(u)(\alpha_{n-1})).
\end{equation}
Then we can write \eqref{eq:q1},  \eqref{eq:systemforQn} as
\begin{equation}\label{eq:eqSystemforQn}
\begin{aligned}
q_1 
&= u + t\alpha_1 + r(t,0,0,u)\\
q_2 
&= u + t\alpha_2 + r(t, \alpha_1, r(t,0,0,u),u)\\
q_3 
&= u + t\alpha_3 + r(t,\alpha_2, r(t,\alpha_1, r(t,0,0,u),u),u)
\end{aligned}
\end{equation}
Now to ease the notation on the higher order terms, let us not write the $u$ base point in the form $r$ above and define
\begin{equation*}\label{eq:defnOfOn}
\begin{aligned}
o_n(t) &:=
\begin{cases}
r(t,0,0) &: \text{if $n=1$}\\
 r(t, \alpha_{n-1}, r(t, \alpha_{n-2}, r(t,  \alpha_{n-3}, ...., r(t,\alpha_1, r(t,0,0))...)&: \text{if $n\geq 2$}
 \end{cases}
\end{aligned}
\end{equation*}
and note the recursion
\begin{equation}\label{eq:identityforOnAndRn}
o_n(t) = r(t,\alpha_{n-1},o_{n-1}(t)).
\end{equation}
All of this suggests the following expression for $q_n,$ which is the main result in this subsection.
\begin{proposition}\label{prop:relationBetweenQnAndUn}
Let \ref{itm:PhiDiff} hold. For each $n$, the following equality holds:
\begin{equation}\label{eq:eqForqn}
q_n(t) = u + t\alpha_n + o_n(t)
\end{equation}
where $t^{-1}o_n(t) \to 0$ as $t \to 0^+$ if \ref{item:assPhiDerivativeBounded} holds and where $\alpha_n$, which is defined in \eqref{eq:alphan}, is positively homogeneous in the direction $d$ and satisfies the VI
\begin{align}
\nonumber &\alpha_n \in \mathcal{K}^u(\alpha_{n-1}) : \langle A\alpha_n - d, \alpha_n - \varphi \rangle \leq 0 \qquad \forall \varphi \in \mathcal{K}^u(\alpha_{n-1})\\
&\mathcal{K}^u(\alpha_{n-1}) := \{ \varphi \in V : \varphi \leq \Phi'(u)(\alpha_{n-1})\text{ q.e. on }\mathcal{A}(u) \text{ and } \langle Au-f, \varphi-\Phi'(u)(\alpha_{n-1}) \rangle = 0\}.\label{eq:alphanTestFunctionSpace}
\end{align}
Furthermore, if $u=S(f,u)$ satisfies strict complementarity with respect to $\alpha_{n-1}$ (see Proposition \ref{prop:formulaForSDiff}), i.e., if
\[\mathcal K^u(\alpha_{n-1}) = \mathcal S^u(\alpha_{n-1}) := \{ \varphi \in V : \varphi = \Phi'(u)(\alpha_{n-1})\text{ q.e. on }\mathcal{A}(u) \},\]
then $\alpha_n$ satisfies
\begin{equation*}
\begin{aligned}
\alpha_n \in \mathcal S^u(\alpha_{n-1}) &: \langle A\alpha_n - d, \alpha_n - \varphi \rangle = 0 \qquad \forall \varphi \in \mathcal S^u(\alpha_{n-1}).
\end{aligned}
\end{equation*}
In this case, if \ref{itm:PhiDiffLinearInDirection} also holds, then $\alpha_n$ is linear in $d$ with respect to positive linear combinations.
\end{proposition}
\begin{proof}
 Let us prove this by induction. The statement \eqref{eq:eqForqn} clearly holds for $n=1$ by \eqref{eq:eqSystemforQn}. Suppose it holds for $n=k$:
\[q_k(t) = u + t\alpha_k + o_k(t).\]
Then we see that
\begin{align*}
q_{k+1}(t) &= S(f+td, q_k(t))\\
&= S(f+td, u + t\alpha_k + o_k(t))\\
&= S(f,u) + t\left(\Phi'(u)(\alpha_k) + \partial S(f,u)[d-A\Phi'(u)(\alpha_k)]\right) + r(t, \alpha_k, o_k(t))\tag{by Proposition \ref{prop:formulaForSDiff}}\\
&= S(f,u) + t\left(\Phi'(u)(\alpha_k) + \delta_{k+1}\right) + r(t,\alpha_k, o_k(t))\tag{by \eqref{eq:deltaNandAlphaN}}\\
&= u +t\alpha_{k+1} + o_{k+1}(t)\tag{by \eqref{eq:idForAlphaN} and \eqref{eq:identityforOnAndRn}}.
\end{align*}
Since $o_k(t) = r(t,u,\alpha_{k-1},o_{k-1}(t))$, the same argument as in the proof of Proposition \ref{prop:formulaForSDiff} shows the vanishing behaviour of the higher order term. We have shown the inductive step and thus the formula holds for each $n$. 
%
From \eqref{eq:defnDeltan} and \eqref{eq:VIForS}, we see that $\delta_n \in \mathcal{K}^u$ satisfies
\begin{equation}\label{eq:VIfordelta}
\begin{aligned}
&\langle A\delta_{n} - d +  A
\Phi'(u)[\Phi'(u)[...\Phi'(u)[\Phi'(u)(\delta_0)+\delta_1] +\delta_2...] + \delta_{n-1}]
,  \delta_{n} - v\rangle \leq 0 \quad \forall v \in \mathcal{K}^u,\\
&\mathcal{K}^u := \{ \varphi \in V : \varphi \leq 0 \text{ q.e. on } \mathcal{A}(u) \text{ and } \langle Au-f, \varphi \rangle = 0\}
\end{aligned}
\end{equation}
and if $u-\Phi(u)=-S_0(A\Phi(u)-f)$ satisfies strict complementarity (see \eqref{eq:blah11}), so that the critical cone simplifies to the linear subspace 
\begin{equation}\label{eq:KScriticalcone}
\mathcal K^u = \mathcal S^u := \{ \varphi \in V : \varphi = 0 \text{ q.e. on } \mathcal{A}(u) \},
\end{equation}
then for all $v \in  \mathcal S^u$, $\delta_n \in \mathcal{S}^u$  solves
\begin{equation}\label{eq:VEfordelta}
\begin{aligned}
\langle A\delta_n - d + A
\Phi'(u)[\Phi'(u)[...\Phi'(u)[\Phi'(u)(\delta_0)+\delta_1] +\delta_2...] + \delta_{n-1}], \delta_n - v \rangle = 0.
\end{aligned}
\end{equation}
The VI \eqref{eq:VIfordelta} can be written as
\begin{equation}\label{eq:VIforalphan}
\langle A\alpha_n - d, \alpha_n - \Phi'(u)[\alpha_{n-1}]  -v \rangle \leq0 \quad \forall v \in \mathcal{K}^u.
\end{equation}
It is not too difficult see from \eqref{eq:defnDeltan} that $\delta_n$ inherits positive homogeneity from $S_0'$, and $\alpha_n$ in turn inherits positive homogeneity from the $\delta_n$ (see \eqref{eq:alphan}).

Under strict complementarity a formula similar to \eqref{eq:VIforalphan} holds (with the inequality changed to equality and the cone replaced by the linear subspace), and $\delta_{0}$ is linear in $d$. From \eqref{eq:defnDeltan} we can see that linearity of $\delta_{n-1}$ in the direction is not enough to guarantee linearity of $\delta_n$; the additional assumption \ref{itm:PhiDiffLinearInDirection} is indeed needed. From this and \eqref{eq:alphan} linearity of the $\alpha_n$ follows easily.
\end{proof}



\begin{remark}\label{rem:LOTsDependOnBasePoint}
In the formula \eqref{eq:eqForqn}, we could also have approximated $u$ by VI iterates $u_n$. Let $u_n := S(f,u_{n-1})$ with $0 \leq u_0 \leq \bar u$ and define $w_n := S(f, w_{n-1})$ with $w_0 := q_0$ and $q_0$ not yet fixed. In lieu of \eqref{eq:eqForqn} we would get
\[q_n(t) = w_n + t\hat \alpha_n + \hat o_n(t)\]
whence we can see that we must choose $q_0=w_0$ such that
\begin{enumerate}
\item $q_n(t) \to q(t)$ for some  $q(t) \in \mathbf{Q}(f+td)$ 
and
\item $w_n \to u$, where $u \in \mathbf{Q}(f)$.
\end{enumerate}
Clearly, these requirements would lead us to an expression for $\mathbf{Q}(f+td)-\mathbf{Q}(f)$ after passing to the limit $n \to \infty$ in the above equality. Pick $q_0 = u_0$: then $w_n=u_n$, but this choice gives us a problem later on that we are not able to handle  --- we cannot show that the higher order term $\hat o_n$ vanishes in the limit $t \to 0^+$ uniformly with respect to the moving base point which is needed to characterise the limiting directional derivative as a derivative. This would have been interesting because we could have chosen $u_0$ to be a sufficiently large upper solution for $\mathbf{Q}(f)$ and then obtained the conical differentiability result for the selection mapping that picks the maximal solution of the QVI. We aim to investigate this further in future works. 
\end{remark}
Now we must pass to the limit $n \to \infty$ in \eqref{eq:eqForqn}. To do this, we need some convergence results for $\alpha_n$ and $o_n(t)$.

\section{Properties of the VI iterates and some estimates}\label{sec:properties}
In this section, we investigate the monotonicity properties of the directional derivatives $\alpha_n$ and the higher order terms $o_n$, and we also study their behaviour on the coincidence set $\mathcal{A}(u)$. The section is concluded with the consideration of a simplification of the VI satisfied by the $\alpha_n$ under some assumptions including the validity of complementarity for \eqref{eq:QVI}.
\subsection{Monotonicity properties of the directional derivatives}
\begin{lemma}\label{lem:monotonicityOfDirDerAndSum}
Let \ref{itm:PhiDiff} and \ref{item:assPhiDerivativeBounded} hold. The sequences
\[\{\alpha_n\}_{n \in \mathbb{N}} \qquad \text{and} \qquad \{t\alpha_n + o_n(t)\}_{n \in \mathbb{N}}\] are monotonically increasing and non-negative almost everywhere on $X$.
\end{lemma}
\begin{proof}
Since the left hand side of $q_n(t) = u+t\alpha_n + o_n(t)$ is increasing in $n$, as is $t\alpha_n + o_n(t)$. Thus
\begin{equation*}\label{eq:monotonicityOfSum}
t\alpha_{n+1} + o_{n+1}(t) \geq t\alpha_n + o_n(t) \geq 0
\end{equation*}
with the second inequality since $q_n(t) \geq u$ (for this we did not need \ref{item:assPhiDerivativeBounded}). 
Dividing by $t$ and sending to zero 
we find using Proposition \ref{prop:relationBetweenQnAndUn} that $\alpha_{n+1} \geq \alpha_n \geq 0$ for all $n \geq 1.$
\end{proof}
By definition, we know that $\delta_n \leq 0$ q.e. on the set $\mathcal A(u)$ (see \eqref{eq:VIfordelta})
which implies that $0 \leq \alpha_n \leq \Phi'(u)(\alpha_{n-1})$ a.e. on $\mathcal A(u)$. We now show that $0 \leq \Phi'(u)(\alpha_n)$ a.e. not only on $\mathcal{A}(u)$ but on the whole of $X$ thanks to the increasing nature of $\Phi$.

\begin{lemma}\label{lem:zeroIsValid}
Under the assumptions in the previous lemma, we have
\[\Phi'(u)(\alpha_n) \geq 0 \quad \text{a.e. on $X$}.\]
\end{lemma}
\begin{proof}
As $\Phi$ is increasing,
\[\Phi'(u)(\alpha_n) = \frac{\Phi(u+t\alpha_n)-\Phi(u)-l(t, \alpha_n,u)}{t} \geq-\frac{l(t, \alpha_n,u)}{t}\]
since the $\alpha_n$ are non-negative a.e. on $X$. 
 Sending $t \to 0^+$ yields the desired conclusion a.e. on $X$. 
\end{proof}
\begin{corollary}Under the assumptions in the previous lemma, we have
\[\alpha_n \geq \delta_{n} \quad \text{a.e. on $X$}\]
\end{corollary}
\begin{proof}
Follows from $\alpha_n = \delta_{n}+\Phi'(u)(\alpha_{n-1})$.
\end{proof}
\subsection{Estimates on the coincidence set}
The next lemma shows that the derivatives $\alpha_n$ vanish on the coincidence set when $\Phi$ is of superposition type.
\begin{lemma}\label{lem:alphaNOnI}Let \ref{itm:PhiDiff} and \ref{item:assPhiDerivativeBounded} hold. We have
\[\alpha_1 = 0 \quad \text{a.e. on $\mathcal{A}(u)$.}\]
Also, if $\Phi$ is a superposition operator, we have for each $n$ that 
\[\alpha_n = 0 \quad\text{ a.e. on $\mathcal{A}(u)$}.\]
\end{lemma}
\begin{proof}
From $t\alpha_n = q_n(t)-u-o_n(t)$, since $q_n \leq \Phi(q_{n-1}) \leq \Phi^2(q_{n-2}) \leq ... \leq \Phi^n(q_0) = \Phi^n(u)$, we find
\begin{equation}\label{eq:4}
0 \leq t\alpha_n \leq \Phi^n(u)-u - o_n(t).
\end{equation}
On the set $\mathcal{A}(u)$, we get $0 \leq t\alpha_1 \leq -o_n(t)$
and dividing here by $t$ and sending to zero, we see by the sandwich theorem that $\alpha_1 = 0$ on $\mathcal{A}(u)$. If $\Phi$ is a superposition operator, observe that if $x$ is such that $u(x) = \Phi(u(x))$, then in fact 
\begin{equation}\label{eq:fact}
u(x) = \Phi^m(u(x))\quad\text{for any $m \in \mathbb{N}$}
\end{equation}
(that is $\Phi$ composed with itself $m$ times). Using this fact on the right hand side of \eqref{eq:4} gives us the result.
\end{proof}
\begin{corollary}\label{cor:deltan}Under the assumptions in the previous lemma, if $\Phi$ is a superposition operator, then
\[\delta_n = 0 \quad\text{a.e. on $\mathcal{A}(u).$}\]
\end{corollary}
\begin{proof}
This follows from $\alpha_n = \Phi'(u)(\alpha_{n-1}) + \delta_{n}$ and Lemma \ref{lem:alphaNOnI}.
\end{proof}\begin{lemma}\label{lem:behaviourOfLOTonI}
Under the assumptions in the previous lemma, we have
\[\frac{o_1(t)}{t} = 0 \quad\text{a.e. on $\mathcal{A}(u)$},\] 
and in the superposition case 
\[\frac{o_{n}(t)}{t} =0 \quad\text{a.e. on $\mathcal{A}(u)$}.\]
\end{lemma}
\begin{proof}
In the superposition case, on $\mathcal{A}(u)$, we have $o_n(t) = q_n(t)-u-t\alpha_n = q_n-u(t)$. Since $q_n \geq u$, we have $o_n(t) \geq 0$ but also $o_n(t) \leq \Phi^n(u)-u = 0$ due to \eqref{eq:fact}. Hence $o_n(t) = 0$. In the general case, the above is true only for $n=1$.
\end{proof}
\subsection{Simplification under regularity of the solution}
The duality pairing appearing in the definition of the set $\mathcal{K}^u(\alpha_{n-1})$ in \eqref{eq:alphanTestFunctionSpace} restricts the class of feasible test functions due to the inconvenient term $\Phi'(u)(\alpha_{n-1})$ which arises because the problem is quasi-variational in nature. Here we shall show that if we assume some regularity then this term can essentially be removed and thus $\mathcal{K}^u(\alpha_{n-1})$ can be simplified. We introduce the following hypotheses, which we use in some of what follows.
\begin{enumerate}[label=(\textbf{C\arabic*})]
\item\label{itm:complementarity} Let $Au, f \in H$ and $(Au-f)(u-\Phi(u))=0$ a.e. in $X$ (i.e., complementarity holds for \eqref{eq:QVI}).
\item\label{itm:propertyOfDeriv} Let $\Phi'(u)(\alpha_n) \geq 0$ q.e. on $\mathcal{A}(u)$ and
$\Phi'(u)(\alpha_n) = 
0$ a.e. in $\mathcal{A}(u)$.
\end{enumerate}
If $\Phi$ is a superposition type operator, then $\Phi'(u)(\alpha_{n})(x) = \Phi'(u(x))\alpha_{n}(x)$ 
so the second part of \ref{itm:propertyOfDeriv} holds by Lemma \ref{lem:alphaNOnI}.


\begin{remark}[Regularity for the QVI problem]\label{rem:regularity}
Let us consider the $H^1(\Omega)$ setting like in the example on the bounded domain $\Omega$ in \S \ref{sec:examples}. The condition \ref{itm:complementarity} appears hard to check in general as it needs $H^2$ and $C^0$ regularity of the solution $u$. For the continuity, we can argue as follows like in \cite{LionsBensoussan}. Let us suppose that $\Phi\colon C^0(\bar \Omega) \to C^0(\bar \Omega)$, \ref{itm:PhiConcaveEtc} and $f \in L^\infty_+$ hold. Define $u_n = S(f, u_{n-1})$ with $u_0 = 0$, which we know converges to the solution $u \in \mathbf{Q}(f)$ in $L^\infty(\Omega)$ and weakly in $V$. By the assumption on $\Phi$, we find that $u_1 \in C^0(\bar \Omega)$, since the solution of the obstacle problem with continuous obstacle is also continuous \cite[\S 2, Corollary 5.3]{LionsBensoussan}.
Thus $u_n \in C^0(\bar \Omega)$ for all $n$, and the convergence $u_n \to u$ in $L^\infty$ implies that $u \in C^0(\bar \Omega)$. 

Regularity in $H^2$ appears much more involved in general. If $\Phi$ carries $H^2(\Omega)$ into $H^2(\Omega)$ we obtain the regularity $u_n \in H^2(\Omega)$ by \cite{BrezisStamp}. Furthermore, the following estimate holds \cite[\S 5, Proposition 2.2]{Rodrigues}:
\begin{align}
\nonumber \norm{Au_n}{H} &\leq \norm{f}{H} + \norm{(A\Phi(u_{n-1})-f)^+}{H}\\
&\leq 2\norm{f}{H} + \norm{A\Phi(u_{n-1})}{H},\label{eq:ctsDependenceRod}
\end{align}
so we need additional assumptions on $\Phi$ to deduce boundedness of the sequence $\{Au_n\}_{n \in \mathbb{N}}$ and hence the regularity $u \in H^2(\Omega)$. As an example, consider the case $\Phi = (-\Delta)^{-1}$ which we discussed before with $A=-\Delta + I$. Then 
\[A\Phi(u_{n-1}) = A(-\Delta)^{-1}u_{n-1} = u_{n-1}+(-\Delta)^{-1}u_{n-1}\]
and the continuous dependence result for elliptic PDEs gives, after bounding the $u_{n-1}$ in norm by the right hand side data $f$,
\[\norm{A\Phi(u_{n-1})}{H} \leq C\norm{f}{H}.\]
Plugging this into \eqref{eq:ctsDependenceRod} leads to a successful resolution of $H^2$-regularity.

See \cite{LionsBensoussan,Mosco1977} for some alternative techniques for regularity (for special cases of $\Phi$) which maybe adaptable to a more general setting.
  \end{remark}
\begin{remark}\label{rem:propertyOfDeriv}
Since $0 \leq \alpha_n =(q (t)-u-o_n(t))\slash{t}\leq (\Phi^n(u)-u-o_n(t))\slash{t}$, if $\Phi'(u)(\cdot)$ is increasing, then the equality assumption of \ref{itm:propertyOfDeriv} holds if
\[\lim_{t \to 0^+}\Phi'(u)\left[\frac{\Phi^n(u)-u-o_n(t)}{t}\right](x) \leq 0\qquad\text{for a.e. $x \in \mathcal{A}(u)$}.\]
\end{remark}
Let us recall the following facts about quasi-everywhere defined sets on a bounded domain $\Omega$ \cite{Wachsmuth}, \cite[\S 6.4.3]{Bonnans}, \cite[\S 8.6]{Delfour}, \cite[\S 2.1]{FukushimaBook}:
\begin{enumerate}
\item  If $\Omega_0 \subset \Omega$ is quasi-open and $v\colon \Omega \to \mathbb{R}$ is quasi-continuous, then $v \geq 0$ a.e. on $\Omega_0$ implies that $v \geq 0$ q.e. on $\Omega_0$.
\item Every element of $V$ has a quasi-continuous representative.
\item $\Omega$ is quasi-open.
\item $\Omega_0 \subset \Omega$ is quasi-closed if $\Omega \setminus \Omega_0$ is quasi-open.
\item The set $\{u < \Phi(u)\}$ is quasi-open. Hence $\Omega \setminus \{u < \Phi(u)\}$ is quasi-closed.
\item If $\Omega_0$ is an open set and $f_1=f_2$ a.e. on $\Omega_0$ then $f_1=f_2$ q.e. on $\Omega_0$. 
\end{enumerate}
So, in the situation of $V=H^1_0(\Omega)$, \ref{itm:propertyOfDeriv} does not necessarily imply that $\Phi'(u)(\alpha_n)=0$ q.e. on $\mathcal{A}(u)$ since $\mathcal{A}(u)$ is not open.

\begin{lemma}\label{lem:rewriteOfConstraintSet}
Let \ref{itm:PhiDiff}, \ref{itm:complementarity} and\ref{itm:propertyOfDeriv} hold.  
Then the set $\mathcal{K}^u(\alpha_{n-1})$ defined in \eqref{eq:alphanTestFunctionSpace} 
can be written as
\[\mathcal{K}^u(\alpha_{n-1}) = \left\{ \varphi \in V : \varphi \leq \Phi'(u)(\alpha_{n-1})\text{ q.e. on }\mathcal{A}(u)\text{ and }  \int_{\mathcal{A}(u)} (Au-f)\varphi= 0\right\}.\]
\end{lemma}
\begin{proof}
Using the complementarity condition, we see that the left hand side of the equality constraint in \eqref{eq:alphanTestFunctionSpace} becomes
\begin{align*}
&\langle Au-f, \varphi-\Phi'(u)(\alpha_{n-1}) \rangle\\
&\quad= \int_{X} (Au-f)(\varphi-\Phi'(u)(\alpha_{n-1}))\\
&\quad= \int_{\mathcal{A}(u)} (Au-f)(\varphi-\Phi'(u)(\alpha_{n-1})) +\int_{\{u< \Phi(u)\}} (Au-f)(\varphi-\Phi'(u)(\alpha_{n-1})) \\
&\quad= \int_{\mathcal{A}(u)} (Au-f)(\varphi-\Phi'(u)(\alpha_{n-1})) \\
&\quad= \int_{\mathcal{A}(u)} (Au-f)\varphi
\end{align*}
where for the final equality we have used the a.e. condition of \ref{itm:propertyOfDeriv}.
\end{proof}

\section{Uniform estimates and passage to the limit in the expansion formula}\label{sec:uniformEstimates}
We look for uniform bounds on $\{\alpha_n\}_{n \in \mathbb{N}}$ in order to deduce that it has a limit. Note the following result which is a direct consequence of the equality \eqref{eq:eqForqn} and the fact that $q_n(t)$ converges to $q(t)$.
\begin{lemma}\label{lem:convOfSum}
Let \ref{itm:PhiDiff} hold. For every $t > 0$, we have $t\alpha_n + o_n(t) \weaklyto \chi(t)$ in $V$ for some $\chi(t)$ as $n \to \infty$. The convergence is strong if \ref{itm:compContOfPhi} is true and is strong in $L^\infty(\Omega)$ if \ref{itm:PhiConcaveEtc} is true and $f,d \in L^\infty_+(\Omega)$.
\end{lemma}
\subsection{Convergence of the directional derivatives}
\begin{theorem}\label{thm:boundednessOfDirectionalDerivatives}
Let \ref{itm:PhiDiff} hold. If either \ref{itm:complementarity} and \ref{itm:propertyOfDeriv} hold, or
\[\exists c > 0 : \quad\norm{\Phi'(u)b}{V} \leq \frac{C_a-c}{C_b}\norm{b}{V}\]
(which is implied by \ref{itm:smallnessOfDerivOfPhi}, see \eqref{eq:consequenceOfA5}), then $\{\alpha_n\}_{n \in \mathbb{N}}$ and $\{\delta_n\}_{n \in \mathbb{N}}$ are bounded in $V$.
\end{theorem}
\begin{proof}
In the first case, the constraint set in the VI \eqref{eq:VIforalphan} for $\alpha_n$ simplifies by Lemma \ref{lem:rewriteOfConstraintSet} and by \ref{itm:propertyOfDeriv}, the  function $\varphi = 0$ is an admissible test function, and this easily leads to a bound on the directional derivatives $\alpha_n$.

Let us discuss the second case now. Abbreviate $N:= (C_a-c)\slash C_b$, pick $v= 0$ in the VI \eqref{eq:VIforalphan} for $\alpha_n$ and use the hypothesis of the lemma to obtain
\begin{align*}
C_a\norm{\alpha_n}{V}^2 &\leq C_b\norm{\alpha_n}{V}\norm{\Phi'(u)[\alpha_{n-1}]}{V} + \norm{d}{V^*}(\norm{\Phi'(u)[\alpha_{n-1}]}{V} + \norm{\alpha_n}{V})\\
&\leq NC_b\norm{\alpha_n}{V}\norm{\alpha_{n-1}}{V} + \norm{d}{V^*}\left(N\norm{\alpha_{n-1}}{V} + \norm{\alpha_n}{V}\right).
\end{align*}
With 
\[K_1 := NC_b, \qquad K_2:=N\norm{d}{V^*},\qquad K_3 := \norm{d}{V^*},\qquad\text{and}\qquad a_n := \norm{\alpha_n}{V},\] the above inequality can be rewritten and manipulated with Young's inequality with $\epsilon$ and $\rho$ as follows:
\begin{align*}
C_aa_n^2 &\leq K_1a_na_{n-1} + K_2a_{n-1} + K_3a_n\\
&\leq K_1\left(\frac{a_n^2}{2} + \frac{a_{n-1}^2}{2}\right) + C_\epsilon K_2^2 + \epsilon a_{n-1}^2 + C_\rho K_3^2 + \rho a_n^2\\
&= \left(\frac{K_1}{2} + \rho\right)a_n^2 + \left(\frac{K_1}{2}+ \epsilon\right) a_{n-1}^2 + C_\epsilon K_2^2  + C_\rho K_3^2,
\end{align*}
implying
\[\underbrace{\left(C_a-\frac{K_1}{2} - \rho\right)}_{=:  C_1}a_n^2 \leq \underbrace{\left(\frac{K_1}{2}+ \epsilon\right)}_{=:  C_2} a_{n-1}^2 + \underbrace{C_\epsilon K_2^2  + C_\rho K_3^2}_{=: C_3}\]
which we write as
\[a_n^2 \leq \frac{ C_2}{ C_1}a_{n-1}^2 + \frac{ C_3}{ C_1}.\]
This implies
\[a_n^2 \leq \left(\frac{ C_2}{ C_1}\right)^{n-1}a_1^2 + \frac{ C_3}{ C_1}\left(\frac{1-(\frac{ C_2}{ C_1})^{n-1}}{1-\frac{ C_2}{ C_1}}\right),\]
so we need $ C_2\slash  C_1 < 1$ for the second term to be bounded uniformly in $n$. That is, we need
\[K_1 < {C_a -\rho - \epsilon},\]
and the right hand side is largest when $\rho$ and $\epsilon$ are small. Thus we need the condition $K_1 < C_a$ or equivalently $N < C_a\slash C_b$ which holds since $c>0$. We then have to choose $\rho$ and $C_\epsilon$ so that the displayed inequality above the previous one is valid. Under this condition we obtain
\[a_n^2 \leq a_1^2 + \frac{C_\epsilon K_2^2 + C_\rho K_3^2}{C_a - K_1 -\rho - \epsilon}.\]
Thus the claim follows for $\alpha_n$. Once the bound on $\alpha_n$ is in hand the bound on the $\delta_n$ follows by testing the VI \eqref{eq:VIfordelta} with zero.
\end{proof}
Hence, we can find a subsequence of the $\alpha_n$ and $\delta_n$ such that $\alpha_{n_j} \weaklyto \alpha$ and $\delta_{n_j} \weaklyto \delta$ for some $\alpha$ and $\delta.$ In fact the convergences hold for the full sequences as the next lemma demonstrates.
\begin{lemma}Under the conditions of Theorem \ref{thm:boundednessOfDirectionalDerivatives}, the full sequence $\alpha_n$ converges weakly 
\[\alpha_n \weaklyto \alpha \quad \text{in $V$},\]
and if \ref{itm:compContOfDerivOfPhi} holds, then $\delta_n$ also converges weakly:
\[ \delta_n \weaklyto \delta \quad \text{in $V$}.\]
\end{lemma}
\begin{proof}
Since the $\alpha_n$ are monotone increasing, they have a pointwise a.e. monotone limit which must agree with $\alpha$ so indeed $\alpha_n \weaklyto \alpha$.  We can pass to the limit in \eqref{eq:idForAlphaN}, which is $\alpha_{n+1}=\Phi'(u)(\alpha_{n}) - \delta_{n+1}$, (\ref{itm:compContOfDerivOfPhi} is sufficient but not necessary; it would also suffice if $\Phi'(u)(\cdot)$ is weak-weak continuous) to find that the weak convergence of the full sequence $\delta_n$.
\end{proof}
As a precursor to characterising the directional derivative $\alpha$, we study the limit $\delta$ in the next lemma.
\begin{lemma}\label{lem:strongConvergenceOfDirDerv}Under \ref{itm:compContOfDerivOfPhi} and the conditions of Theorem \ref{thm:boundednessOfDirectionalDerivatives}, $\delta$ satisfies
\[\delta \in \mathcal{K}^u : \quad\langle A\delta -d +  A\Phi'(u)(\alpha), \delta - \varphi \rangle \leq 0 \qquad \forall v \in \mathcal{K}^u\]
and $\delta_n \to \delta$ and $\alpha_n \to \alpha$ strongly in $V$. If strict complementarity holds for $u-\Phi(u)$ (i.e., \eqref{eq:KScriticalcone} is true) then $\delta$ satisfies
\[\delta \in  {\mathcal S}^u : \quad\langle A\delta - d +  A\Phi'(u)(\alpha), \delta - \varphi \rangle = 0 \qquad \forall v \in  {\mathcal S}^u.\]
\end{lemma}
\begin{proof}
Recall from \eqref{eq:VIfordelta} that $\delta_n$ satisfies the VI
\begin{align*}
\langle A\delta_n -d + A\Phi'(u)(\alpha_{n-1}) , \delta_n - v \rangle &\leq 0 \qquad \forall v \in \mathcal{K}^u
\end{align*}
and so if \ref{itm:compContOfDerivOfPhi} holds, we can pass to the limit here and we see that the limiting object $\delta$ satisfies the VI given in the lemma. We must check that $\delta \in \mathcal{K}^u$ too; this is worth emphasis because of the non-trival quasi-everywhere constraint in $\mathcal{K}^u$. By Mazur's lemma, a convex combination $v_k$ of the $\{\delta_n\}_n$ converges strongly to $\delta$. By definition, $\delta_n \geq 0$ everywhere on $\mathcal{A}(u)\slash \mathcal{A}_n(u)$ where $\mathcal{A}_n(u) \subset \mathcal{A}(u)$ is a set of capacity zero. Since $v_k \to \delta$ strongly in $V$, it holds converges pointwise q.e., and using the fact that a countable union of capacity zero sets has capacity zero, we can pass to the limit to deduce that $\delta \geq 0$ quasi-everywhere on $\mathcal{A}(u)$. Secondly, it is easy to check the equality constraint is satisfied by $\delta$.

Thus taking $v=\delta$ in the VI for $\delta_n$ and $\varphi = \delta_n$ in the VI for $\delta$ and combining, we find
\[C_a\norm{\delta_n - \delta}{V} \leq C_b\norm{\Phi'(u)(\alpha_{n-1}) - \Phi'(u)(\alpha)}{V}\]
and the complete continuity and the weak convergence of $\alpha_n$ to $\alpha$ in $V$ proves the first stated convergence result. For the second, use the identity \eqref{eq:idForAlphaN}: $\alpha_n = \Phi'(u)(\alpha_{n-1}) + \delta_{n}$ and pass to the limit here. If strict complementarity holds for $u-\Phi(u)$ (see \eqref{eq:KScriticalcone}), the variational equality \eqref{eq:VEfordelta} is valid and we can pass to the limit in it since the strong convergence for $\delta_n$ is already obtained. 
\end{proof}
\subsection{Characterisation of the directional derivative of the QVI}
Let us now prove the properties of the directional derivative $\alpha$ stated in Theorems \ref{thm:main} and \ref{thm:regularity}. From \eqref{eq:idForAlphaN}, we obtain
\[\alpha = \Phi'(u)(\alpha)+ \delta.\]
Using this fact in the QVI for $\delta$ given in Lemma \ref{lem:strongConvergenceOfDirDerv} yields
\[\langle A\alpha - d, \alpha - \Phi'(u)(\alpha)-v \rangle \leq0 \quad \forall v \in \mathcal{K}^u
\]
 which translates into, with $\eta := \Phi(u)(\alpha) + v$,
\begin{align*}
\alpha \in \mathcal{K}^u(\alpha) &: \langle A\alpha - d, \alpha - \eta \rangle \leq 0  \quad \forall \eta \in \mathcal{K}^u(\alpha)\\
\mathcal{K}^u(\alpha) &:= \{ \varphi \in V : \varphi \leq \Phi'(u)(\alpha) \text{ q.e. on } \mathcal{A}(u) \text{ and } \langle Au-f, \varphi-\Phi'(u)(\alpha) \rangle = 0\}.
\end{align*}
Similar manipulations lead to the quasi-variational equality for $\alpha$ in the case of strict complementarity. For any valid directions $d_1, d_2$, by Proposition \ref{prop:relationBetweenQnAndUn} we have
\[\alpha_n(c_1d_1) = c_1\alpha_n(d_1)\qquad\forall c_1 \geq 0\]
and if strict complementarity and \ref{itm:PhiDiffLinearInDirection} hold,
\[\alpha_n(c_1d_1 + c_2d_2) = c_1\alpha_n(d_1) + c_2\alpha_n(d_2)\qquad \forall c_1, c_2 \ge 0.\]
Passing to the limit here and above shows that the derivative $\alpha$ is positively homogeneous, and also linear for non-negative directions $d_1, d_2$ and positive constants $c_1, c_2$ under the corresponding assumptions.

\subsection{Convergence of the higher order term}
In this subsection, we need the conditions of Theorems \ref{thm:convergenceOfIterations} and \ref{thm:boundednessOfDirectionalDerivatives}. Since $q_n(t)-u = t\alpha_n + o_n(t) \weaklyto  q(t)-u$ in $V$ and $\alpha_n \weaklyto \alpha$ in $V$ , we immediately obtain the existence of a function $o^*(t) \in V$ such that
\begin{align*}
o_n(t) &\weaklyto  o^*(t)\quad\text{in $V$}
\end{align*}
and it remains for us to show that $o^*$ is a higher order term, i.e., that $t^{-1}o^*(t) \to 0$ as $t \to 0^+$. To do this, in this section we will obtain some convergence (in $t$) results for $o_n(t)$ uniformly in $n$. First, let us give some immediate properties of the limiting object $o^*$.
\begin{lemma}
If \ref{item:assPhiDerivativeBounded} holds, then
\[\liminf_{t \to 0^+}\frac{o^*(t)}{t} \geq 0 \quad\text{a.e. in $X$}.\]
\end{lemma}
\begin{proof}
If $N > n$, we have from Lemma \ref{lem:monotonicityOfDirDerAndSum} that $t\alpha_N + o_N(t) \geq t\alpha_n + o_n(t)$, and taking $N \to \infty$, we obtain from the strong convergence of $\alpha_N + o_N$ in $H$ or $L^\infty_+(\Omega)$ (see Lemma \ref{lem:convOfSum}) that
$t\alpha + o^*(t) \geq t\alpha_n + o_n(t)$, which implies
\[\frac{o^*(t)}{t} \geq \alpha_n - \alpha + \frac{o_n(t)}{t},\]
and taking the limit inferior as $t \to 0^+$ first, using from Proposition \ref{prop:relationBetweenQnAndUn} the fact that $t^{-1}o_n(t) \to 0$ as $t \to 0^+$, and then sending $n \to \infty$ we find the desired result.
\end{proof}
From Lemma \ref{lem:behaviourOfLOTonI} and the pointwise a.e. convergence of $q_n$ and $\alpha_n$ (recall that $\{\alpha_n\}_{n \in \mathbb{N}}$ is a monotonic sequence) we get the following result.
\begin{corollary}
If $\Phi$ is a superposition operator,
\[\frac{o^*(t)}{t} =0 \quad\text{a.e. on $\mathcal{A}(u)$}.\]
\end{corollary}
Now we focus on the final ingredient necessary to prove the main theorem.
\paragraph{Notation}{Since the base point $u$ is fixed, below we have omitted it from the higher order terms (eg. instead of $\hat l(t,a, b, u)$ we just write $\hat l(t,a, b)$ and so on).}

\begin{lemma}\label{lem:convergenceForOn}
In addition to \ref{itm:PhiDiff}, let also the conditions in Theorems \ref{thm:convergenceOfIterations} and \ref{thm:boundednessOfDirectionalDerivatives} as well as \ref{itm:compContOfDerivOfPhi}, \ref{item:assPhiDerivativeBounded} and \ref{itm:smallnessOfDerivOfPhi} hold. 
Then the convergence $t^{-1}o_n(t) \to 0$ in $V$ as $t \to 0^+$ is uniform in $n$.
\end{lemma}
\begin{proof}
The proof is in three steps.
\noindent \paragraph{Step 1} Let us first show by induction that each $q_n(t)$ belongs to $B_{\epsilon \slash 2}(u)$, the closed ball of radius $\epsilon \slash 2$ around $u$. 
 For convenience, let $C_X := C_bC_a^{-1}$, which  is such that $C_\Phi C_X < 1$.
Fix an $\epsilon > 0$ and take 
\begin{equation}\label{eq:assOnT}
t \leq \frac{C_a(1-C_\Phi C_X)\epsilon}{2\norm{d}{V^*}}.
\end{equation}
We have the estimate
\begin{align*}
\norm{q_1(t)-u}{V} &\leq C_a^{-1}t\norm{d}{V^*} \leq \frac\epsilon 2,
\end{align*}
i.e., $q_1(t) \in B_{\epsilon\slash 2}(u).$ 
Suppose the claim holds for the $(n-1)$th term. Regarding $q_n(t)$, we estimate by testing the inequality for $q_n(t)$ with $u-\Phi(y) + \Phi(q_{n-1}(t))$ and the inequality for $u$ with $q_n(t)-\Phi(q_{n-1}(t))+\Phi(u)$:
\begin{align*}
\langle Aq_n(t) - (f+td), q_n(t)-u + \Phi(u)-\Phi(q_{n-1}(t))\rangle &\leq 0,\\
\langle Au - f, y-q_n(t) + \Phi(q_{n-1}(t))-\Phi(u)\rangle &\leq 0,
\end{align*}
whence adding, we obtain
\begin{align*}
\langle A(q_n(t)-u)-td, q_n(t)-u + \Phi(u)-\Phi(q_{n-1}(t)) \rangle \leq 0.
\end{align*}
This leads to 
\begin{align*}
C_a\norm{q_n(t)-u}{V}^2 &\leq C_b\norm{q_n(t)-u}{V}\norm{\Phi(u)-\Phi(q_{n-1}(t))}{V} + t\norm{d}{V^*}\norm{q_n(t)-u}{V}\\
&\quad + t\langle d, \Phi(u)-\Phi(q_{n-1}(t))\rangle \\
&\leq C_b\norm{q_n(t)-u}{V}\norm{\Phi(u)-\Phi(q_{n-1}(t))}{V} + t\norm{d}{V^*}\norm{q_n(t)-u}{V}
\end{align*}
since $d \geq 0$ and $y \leq q_{n-1}(t)$ and $\Phi$ is increasing, giving the bound
\begin{align*}
\norm{q_n(t)-u}{V} &\leq \frac{C_b}{C_a}\norm{\Phi(u)-\Phi(q_{n-1}(t))}{V} + \frac{t}{C_a}\norm{d}{V^*}.
\end{align*}
Now, recalling $C_X$ and using the mean value theorem, we get
\begin{align*}
\norm{q_n(t)-u}{V} 
&\leq C_a^{-1}t\norm{d}{V^*} + C_X\sup_{\lambda \in (0,1)}\norm{\Phi'(\lambda q_{n-1}(t) + (1-\lambda)u)(q_{n-1}(t)-u)}{V}\\
&\leq C_a^{-1}t\norm{d}{V^*} + C_XC_\Phi \norm{q_{n-1}(t)-u}{V}\tag{applying the assumption thanks to the induction hypothesis}\\
&\leq C_a^{-1}t\norm{d}{V^*}(1 + C_XC_\Phi + (C_XC_\Phi)^2 + ... + (C_XC_\Phi)^{n-1})\\
&\leq \frac{C_a^{-1}t\norm{d}{V^*}}{1-C_XC_\Phi }  \tag{by the formula for a geometric series}\\
&\leq \frac\epsilon 2,
\end{align*}
and hence $q_n(t) \in B_{\epsilon\slash 2}(u)$ for all $n$ as long as $t$ satisfies \eqref{eq:assOnT}. 
Observe that for $\lambda \in (0,1)$.
\begin{align*}
\norm{\lambda q_n(t) + (1-\lambda)(u+t\alpha_n)-u}{V} &= \norm{\lambda (q_n(t)-u) + (1-\lambda)t\alpha_n}{V}\\
&\leq \frac \epsilon 2 + t\norm{\alpha_n}{V}\\
&\leq \frac \epsilon 2 + tC^*
\end{align*}
where $C^*$ is the uniform bound (see Theorem \ref{thm:boundednessOfDirectionalDerivatives}) on $\{\alpha_n\}$. Thus if $t$ satisfies \eqref{eq:assOnT} and satisfies $t \leq {\epsilon}\slash {2C^*},$
then $\lambda q_n(t) + (1-\lambda)(u+t\alpha_n) \in B_{\epsilon}(u)$ for all $n$.  Since $\Phi'(u+t\alpha_n+\lambda o_n(t))= \Phi'(\lambda q_n(t) + (1-\lambda)(u+t\alpha_n))$, as long as
\begin{equation*}
t \leq \min\left(\frac{C_a(1-C_\Phi C_X)\epsilon}{2\norm{d}{V^*}}, \frac{\epsilon}{2C^*}\right)=: T_0,
\end{equation*}
we can use \ref{itm:smallnessOfDerivOfPhi} to obtain the boundedness
\begin{equation}\label{eq:requiredSmallness}
\norm{\Phi'(u+t\alpha_n + \lambda o_n(t))o_n(t)}{V} \leq C_\Phi\norm{o_n(t)}{V}.
\end{equation}
 Then Lemma \ref{lem:uniformConvergenceOfL} implies the estimate
\[\frac{\norm{\hat l(t, \alpha_{n-1}, \alpha_{n-1} + t^{-1}o_{n-1}(t))}{V}}{t} \leq C_\Phi \frac{\norm{o_{n-1}(t)}{V}}{t} + \frac{\norm{l(t,\alpha_{n-1})}{V}}{t}.\]
\noindent \paragraph{Step 2} By definition of $o_n$ in \eqref{eq:identityforOnAndRn},
\begin{align*}
o_n(t) &= r(t,\alpha_{n-1}, o_{n-1}(t))\\
&=\hat l(t,\alpha_{n-1},\alpha_{n-1}+t^{-1}o_{n-1}(t))\\
&\quad - \hat o(t,A\Phi'(u)(\alpha_{n-1})-d, A\Phi'(u)(\alpha_{n-1})-d + At^{-1}\hat l(t,\alpha_{n-1},\alpha_{n-1}+t^{-1}o_{n-1}(t))),
\end{align*}
and using the estimates of Lemmas \ref{lem:uniformConvergenceOfhatO} and \ref{lem:uniformConvergenceOfL} and the calculation in the previous step, we find
\begin{align*}
\norm{o_n(t)}{V} &\leq C_\Phi\norm{o_{n-1}(t)}{V} + \norm{l(t,\alpha_{n-1})}{V} +C_a^{-1}\norm{A\hat l(t,\alpha_{n-1},\alpha_{n-1}+t^{-1}o_{n-1}(t))}{V^*} + \norm{o(t,A\Phi'(u)(\alpha_{n-1}))}{V}\\
&\leq C_\Phi\norm{o_{n-1}(t)}{V} + \norm{l(t,\alpha_{n-1})}{V} + C_a^{-1}C_b\norm{\hat l(t,\alpha_{n-1},\alpha_{n-1}+t^{-1}o_{n-1}(t))}{V}+ \norm{o(t,A\Phi'(u)(\alpha_{n-1}))}{V}\\
&\leq C_\Phi(1+C_a^{-1}C_b)\norm{o_{n-1}(t)}{V} + (1+C_a^{-1}C_b)\norm{l(t,\alpha_{n-1})}{V} + \norm{o(t,A\Phi'(u)(\alpha_{n-1}))}{V}\tag{using again \eqref{eq:lBound}}\\
&< C\underbrace{\norm{o_{n-1}(t)}{V}}_{=:a_{n-1}(t)} + \underbrace{(1+C_a^{-1}C_b)\norm{l(t,\alpha_{n-1})}{V} + \norm{o(t,A\Phi'(u)(\alpha_{n-1}))}{V}}_{=:b_{n-1}(t)}
\end{align*}
for some $C<1$ by the assumption on $C_\Phi$ in \ref{itm:smallnessOfDerivOfPhi}. The above can be recast as
\[a_n(t) \leq Ca_{n-1}(t) + b_{n-1}(t)\]
which can be solved for $a_n$ in terms of $a_1$ and $\{b_i\}_{i=1}^{n-1}$:
\begin{align}
a_n(t) 
&\leq C^{n-1}a_{1}(t) + C^{n-2}b_1(t) + C^{n-3}b_2(t) + ... + Cb_{n-2}(t) + b_{n-1}(t)\label{eq:prelim2}.
\end{align}
\noindent \paragraph{Step 3}Let us see why
\[\frac{b_{n-1}(t)}{t}:= \frac{(1+C_a^{-1}C_b)\norm{l(t,\alpha_{n-1})}{V}}{t} + \frac{\norm{o(t,A\Phi'(u)(\alpha_{n-1}))}{V}}{t} \to 0\]
uniformly in $n$. By Lemma \ref{lem:strongConvergenceOfDirDerv}, $\{\alpha_{n-1}\}_n$ and $\{A\Phi'(u)(\alpha_{n-1})\}_n$ are both sets that are subsets of compact sets,  
hence, thanks to the compact differentiability of $\Phi$ (from \ref{itm:PhiDiff}) and compact differentiability of $S_0$, for any $\epsilon  > 0$, there exists a $T_1 > 0$ independent of $j$ such that 
\[t \leq T_1 \implies \frac{b_{j}(t)}{t} \leq \frac{(1-C)\epsilon}{2}\qquad \forall j.\]
\noindent \paragraph{Step 4} As $o_1(t) = r(t,0,0)=o(t,d)$ is a higher order term, we know that there is a $T_2 > 0$ such that
\[t \leq T_2 \implies \frac{\norm{o(t,d)}{V}}{t} \leq \frac{\epsilon}{2}.\]
Now recalling \eqref{eq:prelim2}, for $t \leq \min(T_0, T_1,T_2)$,
\begin{align*}
\frac{\norm{o_n(t)}{V}}{t}
&\leq C^{n-1}\frac{\norm{o(t,d)}{V}}{t} +
C^{n-2}\frac{b_1(t)}{t} + C^{n-3}\frac{b_2(t)}{t} + ...  + \frac{b_{n-1}(t)}{t}\\
&\leq \frac{\norm{o(t,d)}{V}}{t} + C^{n-2}\frac{b_1(t)}{t} + C^{n-3}\frac{b_2(t)}{t}+ ...  + \frac{b_{n-1}(t)}{t}\\
&\leq \frac{\epsilon}{2} + \frac{\epsilon(1-C)}{2}\left(C^{n-2} + C^{n-3} + ...  + C + 1\right)\tag{for any $\epsilon > 0$ by the previous step}\\
&= \frac{\epsilon}{2} + \frac{\epsilon(1-C)(1-C^{n-1})}{2(1-C)}\\
&\leq \epsilon
\end{align*}
This shows that $o_n(t)\slash t$ tends to zero uniformly in $n$.
\end{proof}
\subsection{Passing to the limit and conclusion}
Let us now pass to the limit in \eqref{eq:eqForqn} and conclude the main result. Sending $n \to \infty$, using all the convergence results we obtain
\[q(t)= u+ t\alpha + o^*(t).\]
We now prove that $o^*$ is indeed a higher order term. 
\begin{lemma}Let the conditions of Lemma \ref{lem:convergenceForOn} hold. The function $o^*$ satisfies
\[\frac{o^*(t)}{t} \to 0\quad\text{in $V$ as $t \to 0^+$}.\]
\end{lemma}
\begin{proof}
We present two proofs.
\paragraph{1. Without using compactness of $\Phi$}
We proved in Lemma \ref{lem:convergenceForOn} that for any $\epsilon > 0$, there exists a $T>0$ (independent of $n$) such that  if $t \leq T$, then $t^{-1}\norm{o_n(t)}{V} \leq \epsilon$. We can use $o_n(t) \weaklyto o^*(t)$ (see the previous subsection) and the weak lower semicontinuity of norms to deduce that also
\[\frac{\norm{o^*(t)}{V}}{t} \leq \liminf_{n \to \infty} \frac{\norm{o_n(t)}{V}}{t} \leq \epsilon.\]
\paragraph{2. Using compactness of $\Phi$}
In the last section we showed that
\begin{equation}\label{eq:con1}
\alpha_n +  \frac{o_n(t)}{t} \to \alpha_n \quad \text{as $t \to 0^+$ uniformly in $n$ in $V$}.
\end{equation}
Note that $q_n \to q$ strongly in $V$ due to \ref{itm:compContOfPhi}. From the equation \eqref{eq:eqForqn} the $q_n$ satisfies, since $q_n-u \to q-u$, we see that
\begin{equation}\label{eq:con2}
\alpha_n +  \frac{o_n(t)}{t} \to \frac{q(t)-u}{t} \quad \text{as $n \to \infty$ in $V$}.
\end{equation}
Now thanks to \eqref{eq:con1} and \eqref{eq:con2} we can  apply the Moore--Osgood theorem (eg. see \cite[\S I.7, Lemma 6]{DunfordSchwartzLinear}) which tells us that the double limit exists and that we can interchange the order of the limit-taking so that
\begin{align}
\nonumber \lim_{n\to\infty} \alpha_n &= \lim_{n\to\infty} \lim_{t \to 0^+} \left(t^{-1}o_n(t) + \alpha_n\right) = \lim_{t \to 0^+} \lim_{n\to\infty} \left(t^{-1}o_n(t) + \alpha_n\right)\\
&= \lim_{t \to 0^+} \frac{q(t)-u}{t} < \infty.\label{eq:middle}
\end{align}
Then
\[\frac{o_n(t)}{t} = \frac{q_n-u}{t}-\alpha_n \to \frac{q-u}{t}-\alpha \quad\text{as $n \to \infty$}\]
but the LHS weakly converges as $n \to \infty$ to $o^*(t)\slash {t}$, so that
\[\frac{q(t)-u}{t} = \alpha + \frac{o^*(t)}{t}\] and if we take the limit $t \to 0^+$ on this equality and use \eqref{eq:middle}, we find
\[\alpha + \frac{o^*(t)}{t} \to \alpha\quad\text{as $t \to 0^+$}.\]
\end{proof}
This completes the proof of Theorem \ref{thm:main}. It is worth discussing the assumptions \ref{itm:compContOfPhi} and \ref{itm:PhiConcaveEtc} and $f, d \in L^\infty_+(\Omega)$ which were used originally in Theorem \ref{thm:convergenceOfIterations}. Without these assumptions, we would still get the weak convergence of the iterates $q_n(t)$ but we cannot in general identify the weak limit as an element of $\mathbf{Q}(f+td)$. The convergence results of eg. Theorem \ref{thm:boundednessOfDirectionalDerivatives}  (regarding the directional derivatives $\alpha_n$) and Lemma \ref{lem:convergenceForOn} (regarding the higher order terms $o_n$) would still hold. Hence, if a different method is available to identify $q(t)$ then these assumptions can be removed or replaced.

\section{Application to thermoforming}\label{sec:thermoforming}

The aim of thermoforming is to manufacture products by heating a membrane or plastic sheet to its pliable temperature and then forcing the membrane (by means of vacuum or high gas pressure) onto a mould, commonly made of aluminium or some aluminium alloy, which makes the membrane deform and take on the shape of the mould. 

The process is applied to form large structures such as car panels but also to create microscopic products such as microfluidic structures (e.g. channels on the range of micrometers). The amount of applications and the necessity of precision of some of the thermoformed structures has sparked research into its modelling and accurate numerical simulation as can be seen in \cite{Munro2001} and \cite{Warby2003}. 

The contact problem associated with the heated plastic sheet and the mould can also be described as a variational inequality problem assuming perfect sliding of the membrane with the mould as described in \cite{Whiteman2000}. However, a complex phenomenon takes place when the heated sheet is forced into contact with the mould: in principle, the mould is not at the same temperature as the plastic sheet (it might be relatively cold with respect to membrane) which triggers a heat transfer process with difficult-to-predict consequences (see for example \cite{Lee2001}, e.g., it changes the polymer viscosity). In practice, the thickness of the thermoformed piece can be controlled locally by the mould structure and its initial temperature distribution (see \cite{Lee2001}) and the non-uniform temperature distribution of the polymer sheet has substantial changes on the results (see \cite{Nam2000}).

A common mould material is aluminum and the large heat fluctuations create a substantial difference in the size of the mould; aluminum has a relatively high thermal expansion volumetric coefficient and this implies that there is a dynamic change in the obstacle (the mould) as the polymer sheet is forced in contact with it. This determines a compliant obstacle-type problem as the one described in \cite{Outrata1998} and hence the overall process is a QVI with underlying complex nonlinear PDEs determining the heat transfer and the volume change in the obstacle.

In what follows we consider this compliant obstacle behavior whilst simultaneously making various simplifying assumptions in order to study a basic but nevertheless meaningful model.

\subsection{The model}

We restrict the analysis to the 1D case for the sake of simplicity; the results can be extended to the 2D or higher dimensional case with a few modifications. However, we provide 2D numerical tests. Let $\Phi_0\colon [0,1] \to \mathbb{R}^+$ be the (parametrised) mould shape that we wish to reproduce through a sheet (or membrane). The membrane lies below the mould and is pushed upwards through some mechanism (usually vacuum and/or air pressure) denoted as $f$.  
We make the following three simplified fundamental physical assumptions:
\begin{enumerate}
\item The temperature for the membrane is always a constant prescribed value 
\item The mould grows in an affine fashion with respect to changes in its temperature 
\item The temperature of the mould is subject to diffusion, convection and boundary conditions arising from the insulated boundary and it depends on the vertical distance between the mould and the membrane.
\end{enumerate}

Although the thermoforming process is a time evolution process, the setting described by assumptions 1-3 is appropriate for one time step in the time semi-discretization of such process, and thus fits the mathematical framework of the present paper.  

We denote the position of the mould and membrane by $\Phi(u)$ and $u$ respectively and $T$ will stand for the temperature of the mould. Let us define the spaces $W = H^1(0,1)$ and $H=L^2(0,1)$ and let either $A=-\Delta_{N} + I$ and $V=H^1(0,1)$ or $A=-\Delta_D$ and $V=H^1_0(0,1)$ in the case of Neumann or Dirichlet boundary respectively for the membrane $u$ (zero Dirichlet conditions arise from clamping the the membrane at its ends). The system we consider is the following:
\begin{align}
u \in V : u \leq \Phi(u), \quad \langle Au-f, u-v \rangle &\leq 0\quad \forall v \in V : v \leq \Phi(u)\label{eq:QVIforu}\\
kT-\Delta T &= g(\Phi(u)-u) &&\text{on $[0,1]$}\label{eq:pdeForT}\\
\partial_\nu T &= 0 &&\text{on $\{0, 1\}$ }\label{eq:bcForT}\\
\Phi(u) &= \Phi_0 + LT, &&\text{on $[0,1]$}\label{eq:mould}
\end{align}
where 
$f \in H_+$, $k>0$ is a constant, $\Phi_0 \in V$, $L\colon W \to V$ 
is a bounded linear operator such that
\[
\text{for every $\Omega_0 \subset \Omega$, if $u \leq v$ a.e. on $\Omega_0$ then $Lu \leq Lv$ a.e. on $\Omega_0$,}\]
and  $g\colon \mathbb{R} \to \mathbb{R}$ is decreasing and $C^2$ with $g(0)=M > 0$ a constant, $0 \leq g \leq M$ and $g'$ bounded\footnote{Under these circumstances, $g$ maps $W$ into $W$ \cite[Theorem 1.18]{MR1207810}}. Thus when the membrane and mould are in contact or are close to each other, there is a maximum level of heat transfer onto the mould, whilst when they are sufficiently separated, there is no heat exchange. An example of $g$ to have in mind is a smoothing of the function
\begin{equation}\label{eq:exampleg}
g(r) = \begin{cases}
1 &: \text{if $r \leq 0$}\\
1-r &: \text{if $0 < r < 1$}\\
0 &: \text{if $r \geq 1$}
\end{cases}.
\end{equation}
Note that the local increasing property of $L$ stated above is equivalent to
\begin{equation}\label{item:nonnegativityOfL} 
\text{for every $v \in W$, $vLv \geq 0$ a.e.}
\end{equation}
When $V=H^1_0(0,1)$ is chosen 
$L$ must carry $H^1(0,1)$ into $H^1_0(0,1)$ (as assumed); an example of such an $L$ is given by the pointwise multiplication by a smooth bump function or mollifier that approximates the identity and vanishes near the boundary.
This assumption guarantees that  $\Phi\colon H \to V$ in every case. If $V=H^1(0,1)$ then $V \equiv W$. 

The system above is derived as follows. Consideration of the potential energy of the membrane will show that $u$ solves the inequality \eqref{eq:QVIforu} \cite{Whiteman2000} with the novel QVI nature resulting from assuming that heat transfer occurs between the membrane and mould (the membrane modifies the mould and vice versa). If we let $\hat T\colon \Gamma \to \mathbb{R}$ be the temperature of the mould defined on the curve
\[\Gamma :=\{(r,\Phi(u)(r)) : r \in [0,1]\} \subset \mathbb{R}^2\]
(which is a 1D hypersurface in 2D), our modelling assumptions directly imply that $\hat T$ solves the PDE
\begin{equation}\label{eq:pdeForhatT}
\begin{aligned}
k\hat T(x)-\Delta_\Gamma \hat T(x) &= g(x_2-u(x_1)) &&\text{for $x=(r,\Phi(u)(r)) \in \Gamma$}\\
\frac{\partial \hat T}{\partial \nu} &= 0 &&\text{on $\partial\Gamma$ 
}
\end{aligned}
\end{equation}
where the notation $x_i$ means the $i$th component of $x$. We reparametrise by $T(r) = \hat T(r, \Phi(r))$ and simplify the equation in \eqref{eq:pdeForhatT} (namely the Laplace--Beltrami term, see Appendix \ref{sec:curvature}) to obtain \eqref{eq:pdeForT}.


\subsection{First properties and existence for the system}
Plugging in \eqref{eq:mould} into \eqref{eq:pdeForT} we obtain
\begin{equation}\label{eq:equationForT}
\begin{aligned}
kT-\Delta T &= g(LT+\Phi_0 -u) &&\text{on $[0,1]$}\\
\partial_\nu T &= 0 &&\text{on $\{0, 1\}$ }
\end{aligned}
\end{equation}
\begin{lemma}\label{lem:existenceForT}For every (given) $u \in H$, there exists a unique solution $T \in W$ to the equation \eqref{eq:equationForT}.
\end{lemma}
\begin{proof}
Define the nonlinear operator $B$ by $BT := (k-\Delta)T - g(LT + \Phi_0 -u)$. We see that
\begin{align*}
\norm{BT}{W^*} &= \sup_{v \in W}|\langle BT, v \rangle|\\
&\leq \sup_v \max(1,k)\norm{T}{W}\norm{v}{W} + \norm{g}{L^\infty}\norm{v}{L^1}\tag{using the bound on $g$}\\
&\leq \sup_v \max(1,k)\norm{T}{W}\norm{v}{W} + \norm{g}{L^\infty}\norm{v}{W}\\
&=\max(1,k)\norm{T}{W}+ \norm{g}{L^\infty},
\end{align*}
so if $T$ lies in a bounded subset of $W$, as does $BT$, thus $B$ is a bounded operator. It is also coercive since
\begin{align*}
\langle BT, T \rangle 
&\geq \min(1,k)\norm{T}{W}^2 - \norm{g}{L^\infty}\norm{T}{W}
\end{align*}
implies that
\[\frac{\langle BT, T \rangle}{\norm{T}{W}} 
\geq \min(1,k)\norm{T}{W} - \norm{g}{L^\infty}\]
which tends to infinity as $T \to \infty$.

Now take $T_n \weaklyto T$ in $W$. By continuity, $LT_n \weaklyto LT$ in $V$ and this convergence is strong in $H$. Then
\begin{align*}
\norm{g(LT_n+\Phi_0-u)-g(LT+\Phi_0-u)}{H}
&\leq\Lip{g}\norm{LT_n-LT}{H}\\
&\to 0,
\end{align*}
hence we have complete continuity of $g(L(\cdot)+ \Phi_0-u)\colon W \to H$.

The map $(k-\Delta)\colon W \to W^*$ is monotone and hemicontinuous, hence it is of type M \cite[Lemma 2.2]{Showalter}. Since $g(L(\cdot)-u)\colon W \to W^*$ is completely continuous, by \cite[Example 2.B]
{Showalter}, the sum $B$ is of type M. Collecting all of the above facts, we may apply Corollary 2.2 of \cite{Showalter} to obtain the existence.

Uniqueness follows from a continuous dependence result for two solutions for two data:
\begin{align*}
kT_1 - \Delta T_1 &= g(LT_1+\Phi_0-u_1)\\
kT_2 - \Delta T_2 &= g(LT_2+\Phi_0-u_2)\\
\partial_\nu T_1 &= 0\\
\partial_\nu T_2 &= 0.
\end{align*}
Take the difference to find
\begin{align*}
\min(k,1)\norm{T_1-T_2}{W}^2 &= \int (g(LT_1+\Phi_0-u_1)-g(LT_2+\Phi_0-u_2))(T_1-T_2).
\end{align*}
Since $L$ is locally increasing, uniqueness can be inferred directly  using the decreasing property of $g$. 
\end{proof}
\begin{lemma}\label{lem:PhiZeroNonNegative}It holds that $\Phi(0) \geq 0$ a.e.
\end{lemma}
\begin{proof}
Note that $\Phi(0) = \Phi_0 + LT|_{u=0} =: \Phi_0 + LT_0$ where $T_0$ solves \eqref{eq:equationForT} with the right hand side equal to $g(LT_0 + \Phi_0)$, i.e.
\begin{equation*}
\begin{aligned}
kT_0-\Delta T_0 &= g(LT_0+\Phi_0) &&\text{on $[0,1]$}\\
\partial_\nu T_0 &= 0 &&\text{on $\{0, 1\}$}.
\end{aligned}
\end{equation*}
Test this equation with $T_0^-$:
\[k\int |T_0^-|^2 + \int |\nabla T_0^-|^2 = -\int g(LT_0 + \Phi_0)T_0^- \leq 0\]
which immediately implies that $T_0 \geq 0$. The claim follows by the local increasing property of $L$.
\end{proof}
\begin{lemma}\label{lem:PhiIncreasing}
The map $\Phi\colon H \to V \subset H$ is increasing.
\end{lemma}
\begin{proof}
To show that $u \mapsto \Phi(u) = LT(u) + \Phi_0$ is increasing, it suffices to show that $u \mapsto T(u)$ is increasing. Take the solutions $T_1$ and $T_2$ of the equation \eqref{eq:equationForT} corresponding to $u=u_1$ and $u=u_2$ (and these solutions exist by Lemma \ref{lem:existenceForT}) with $u_1 \leq u_2$, 
take the difference of the equations and test with $(T_1-T_2)^+$:
\begin{align*}
&\int k|(T_1-T_2)^+|^2 + |\nabla (T_1-T_2)^+|^2\\
&\quad= \int (g(LT_1+\Phi_0-u_1)-g(LT_2+\Phi_0-u_2))(T_1-T_2)^+\\
&\quad=\int_{\{T_1\geq T_2 \}} (g(LT_1+\Phi_0-u_1)-g(LT_2+\Phi_0-u_2))(T_1-T_2)
\end{align*}
On the area of integration, we have by the local increasing property of $L$ that $LT_1\geq LT_2$, and since $u_1 \leq u_2$, we have $-u_1 \geq -u_2$, which implies that $LT_1+\Phi_0-u_1\geq LT_2+\Phi_0-u_2$ pointwise a.e. Since $g$ is of superposition type and is decreasing, $g(LT_1+\Phi_0-u_1)-g(LT_2+\Phi_0-u_2) \leq 0$, and hence the above integral (whose integrand is the product of non-positive and non-negative terms) is less than or equal to zero. Hence we have that $(T_1-T_2)^+ = 0$ in $H$ giving $T_1 \leq T_2$ on $\Omega$. Applying $L$ to both sides and using the increasing property, we find the result.
\end{proof}
\begin{theorem}\label{thm:existenceForThermoformingSystem}
There exists a solution $(u, T, \Phi(u))$ to the system \eqref{eq:QVIforu}, \eqref{eq:pdeForT}, \eqref{eq:bcForT}, \eqref{eq:mould}.
\end{theorem}
\begin{proof}
By Lemmas \ref{lem:PhiZeroNonNegative} and \ref{lem:PhiIncreasing} and the Tartar--Birkhoff theory, there exists a solution $u$ to the QVI \eqref{eq:QVIforu} --- an explicit expression for $\Phi(u)$ in terms of $u$ is not needed. Now, with this $u$ fixed, apply Lemma \ref{lem:existenceForT} to uniquely determine $T(u)$ and thus $\Phi(u)$, and consequently \eqref{eq:pdeForT} has a solution $T$.
\end{proof}
Now, before we discuss compactness of $\Phi$, let us give the following continuous dependence result for two solutions $T_1, T_2$ corresponding to data $u_1, u_2$:
\begin{align}
\nonumber \min(k,1)\norm{T_1-T_2}{W}^2 &= \int (g(LT_1+\Phi_0-u_1)-g(LT_2+\Phi_0-u_2))(T_1-T_2)\\
\nonumber &\leq \Lip{g}\int |LT_1-LT_2 + u_2 - u_1||T_1-T_2|\\
\nonumber &\leq \Lip{g}\left(\norm{LT_1-LT_2}{H}\norm{T_1-T_2}{H} + \norm{u_1 - u_2}{H}\norm{T_1-T_2}{H}\right)\\
&\leq \Lip{g}\left(\norm{L}{\mathcal{L}(W, H)}\norm{T_1-T_2}{W}^2+ \norm{u_1 - u_2}{H}\norm{T_1-T_2}{H}\right).\label{eq:ctsDependenceForT}
\end{align}
 We use the following hypothesis at various points.
\begin{enumerate}[label=(\textbf{X\arabic*})]
\item\label{item:smallnessForF} $\Lip{g}\norm{L}{\mathcal{L}(W,H)} < \min(1,k)$ 
\end{enumerate}
\begin{lemma}If \ref{item:smallnessForF} holds, $\Phi\colon V \to V$ is completely continuous.
\end{lemma}
\begin{proof}
Suppose that $u_n \weaklyto u$ in $V$ and consider the PDEs corresponding to data $u_n$ and $u$:
\begin{align*}
kT_n - \Delta T_n &= g(LT_n+\Phi_0-u_n)\\
kT - \Delta T &= g(LT+\Phi_0-u)\\
\partial_\nu T_n &= 0\\
\partial_\nu T &= 0
\end{align*}
We wish to show that $T_n \to T$. The estimate \eqref{eq:ctsDependenceForT} implies
\begin{align*}
\min(k,1)\norm{T_n-T}{W}^2 &\leq \Lip{g}\left(\norm{L}{\mathcal{L}(W, H)}\norm{T_n-T}{W}^2+ \norm{u - u_n}{H}\norm{T_n-T}{H}\right).
\end{align*}
Thus under the condition in the lemma, we can move the first term on the RHS onto the LHS, divide by $\norm{T_n-T}{}$ and then take the limit to see that $T_n \to T$ in $W$ and by continuity of $L\colon W \to V$ that $\Phi(u_n) \to \Phi(u)$ in $V$. 
\end{proof}
\subsection{Differentiability of $\Phi$}
We want to prove now that $\Phi\colon V \to V$ is differentiable.
\begin{theorem}If $g''$ is bounded from above, the map $\Phi\colon V \to V$ is Fr\'echet differentiable at a solution $u$ given by Theorem \ref{thm:existenceForThermoformingSystem}. 
Furthermore, $-L\delta:=\Phi'(u)(d)$ satisfies the PDE
\begin{equation*}\label{eq:pdeForDelta}
(k-\Delta)\delta - g'(\Phi(u)-u)L\delta = g'(\Phi(u)-u)d.
\end{equation*}
\end{theorem}
\begin{proof}
The idea is to apply the implicit function theorem to the map $\mathcal{F}\colon V\times W \to W^*$ defined by
\[\mathcal{F}(u,T) = kT-\Delta T -g(LT + \Phi_0 -u),\]
which we understand via the duality pairing
\[\langle \mathcal{F}(u,T), \varphi \rangle_{W^*, W} = k\int T\varphi + \int \nabla T \nabla \varphi - \int g(LT+\Phi_0-u)\varphi.\]
In order to do this, we have to check a number of properties.
\paragraph{(1) Fr\'echet differentiability of $\mathcal{F}$} We concentrate on the nonlinear term in $\mathcal{F}$ since the linear terms are clearly differentiable. Since $g \in C^2(\mathbb{R})$, Taylor's theorem gives the expression
\[g(x+h)-g(x) = g'(x)h + \frac 12 g''(\theta)h^2, \qquad x,h \in \mathbb{R}\]
where $\theta \in (x, x+h)$. Since $g''$ is bounded, we estimate
\[|(g(x+h)-g(x)-g'(x)h)w| \leq \frac 12 \norm{g''}{L^\infty}h^2|w|\qquad x,h, w \in \mathbb{R}.\]
This implies
\begin{align*}
\norm{g(v+d)-g(v)-g'(v)d}{W^*} &= \sup_{\substack{w \in W\\ \norm{w}{W}=1}} |\langle g(v+d)-g(v)-g'(v)d, w \rangle_{W^*,W}|\\
&= \sup_{\substack{w \in W\\ \norm{w}{W}=1}} \left|\int (g(v+d)-g(v)-g'(v)d)w\right|\\
&\leq \frac{\norm{g''}{L^\infty}}{2} \sup_{\substack{w \in W\\ \norm{w}{W}=1}} \int d^2|w|\\
&\leq \frac{\norm{g''}{L^\infty}\norm{d}{L^3}^2}{2} \sup_{\substack{w \in W\\ \norm{w}{W}=1}} \norm{w}{L^3} 
\end{align*}
by Holder's inequality with exponents $(3\slash 2, 3)$. When the dimension is less than $6$, $H^1 \subset L^3$ continuously by a Sobolev inequality (see eg. Corollary 9.14 of \cite{MR2759829}) so that the above becomes
\[\frac{\norm{g(v+d)-g(v)-g'(v)d}{W^*} }{\norm{d}{W}}
\leq C\frac {\norm{g''}{L^\infty} \norm{d}{W}}{2}.
\]
Now take the limit $d \to 0$ in $W$ and we see that $g\colon W \to W^*$ is Fr\'echet differentiable. The composition of Fr\'echet differentiable maps is also Fr\'echet, so we have shown that the $\mathcal{F}$ is Fr\'echet with respect to $z$ and $u$. Indeed 
\begin{equation}\label{eq:derivativeOfG}
\mathcal{F}'(u,T)(d,h) = (k-\Delta)h - g'(LT+\Phi_0-u)(Lh-d).
\end{equation}

\paragraph{(2) Continuity of $\mathcal{F}'$}The mean value theorem yields the estimate
\[|g'(x)-g'(y)| \leq |g''(\lambda)||x-y|  \qquad x,y \in \mathbb{R}\]
for some $\lambda \in (x,y)$.  Hence if $v_n \to v$ in $W$, we find
\begin{align*}
\norm{g'(v_n)-g'(v)}{\mathcal{L}(W,W^*)} 
&= \sup_{\substack{w \in W\\\norm{w}{W} = 1}}\left|\int(g'(v_n)-g'(v))w\right|\\
&\leq \norm{g''}{L^\infty}\sup_{\substack{w \in W\\\norm{w}{W} = 1}} \int|v_n-v||w|\\
&\leq \norm{g''}{L^\infty}\norm{v_n-v}{H},
\end{align*}
showing that $g'\colon W \to \mathcal{L}(W,W^*)$ is continuous. By this fact and linearity we get that the full Fr\'echet derivative of the map in question is continuous.

Observe that
\begin{align*}
\partial_{T}\mathcal{F}(u,T)(h) = (k-\Delta)h - g'(LT+\Phi_0-u)Lh
\end{align*}
which is clearly linear in the direction. The map $\partial_{T}\mathcal{F}(u,T)\colon W \to W^*$ is also continuous:
\begin{align*}
\norm{\partial_{T}\mathcal{F}(u,T)(h)}{W^*} &= \norm{(k-\Delta)h - g'(LT+\Phi_0-u)Lh}{W^*} \\
&\leq \max(1,k)\norm{h}{W} + \norm{g'(LT+\Phi_0-u)Lh}{H} \\
&\leq \max(1,k)\norm{h}{W} +  \norm{g'}{\infty}\norm{Lh}{H}\\
&\leq C\norm{h}{W}.
\end{align*}
\paragraph{(3) Invertibility of $h\mapsto \partial_{T}\mathcal{F}(u,T)(h)$} 
We need to show that this derivative is invertible, i.e., for every $b \in W^*$, there exists a $h\in W$ such that $\partial_{T}\mathcal{F}(u,T)(h)= b$ or equivalently
\begin{equation}\label{eq:preEX3}
\begin{aligned}
(k-\Delta)h - g'(LT + \Phi_0 -u)Lh = b.
\end{aligned}
\end{equation}
By Lax--Milgram (applicable since $g' \leq 0$ and we have \eqref{item:nonnegativityOfL}, the second term leads to a coercive bilinear form), this has a unique solution $h \in W$ such that
\[\min(1,k)\norm{h}{W} \leq \norm{b}{W^*}.\]
The inverse is also bounded by the continuous dependence above. For later convenience, we set the solution mapping of the PDE \eqref{eq:preEX3} as $h=\mathcal{S}(b,LT+\Phi_0-u)$. 

\paragraph{(4) Application of the implicit function theorem}Suppose $(u,T)$ is such that $\mathcal{F}(u,T) = 0$ (eg. $(u,T)$ could be a solution given by Theorem \ref{thm:existenceForThermoformingSystem}). The implicit function theorem then gives the existence of a Fr\'echet differentiable map $h\colon N_u \to W$ defined on a neighbourhood $N_u \subset V$ of $u$ such that $h(u)=T$ and $\mathcal{F}(v,h(v))=0$ for all $v \in N_u$, i.e.,
\begin{equation*}
\begin{aligned}
kh(v)-\Delta h(v) -g(Lh(v)+\Phi_0-v)&=0\qquad\forall v \in N_u.
\end{aligned}
\end{equation*}
By definition, $\Phi(u) = \Phi_0 + LT|_u$, and since solutions of the PDE \eqref{eq:pdeForT} for $T$ are unique, we deduce that $T|_u = h(u)$, and so $\Phi$ inherits the differentiability property of $h$, at least locally around $u$.
\paragraph{(5) 
Expression for $\Phi'(u)(h)$}The implicit function theorem tells us that \[h'(u)=-[\partial_{T}\mathcal{F}(u,h(u))]^{-1}\partial_u \mathcal{F}(u,h(u)).\]
Observe from \eqref{eq:derivativeOfG} that 
\begin{equation}\label{eq:preEX1}
\partial_u \mathcal{F}(u,h(u))(v) = g'(Lh(u)+\Phi_0-u)v,
\end{equation}
and by recalling the definition of the solution mapping $\mathcal{S}$ associated to \eqref{eq:preEX3} above we see that 
\begin{equation}\label{eq:preEX2}
[\partial_{T}\mathcal{F}(u,h(u))]^{-1}(b) =
\mathcal{S}(b, Lh(u) + \Phi_0 -u ).
\end{equation}
Combining \eqref{eq:preEX1} and \eqref{eq:preEX2},
\begin{align*}
h'(u)(d) 
&= -\mathcal{S}(g'(Lh(u)+\Phi_0-u)d,Lh(u)+\Phi_0-u),
\end{align*}
which implies that since $\Phi(u) = \Phi_0 + LT|_u = \Phi_0 + Lh(u)$,
\begin{align*}
\Phi'(u)(d) &= L(h'(u)(d))\\
 &= L(-\mathcal{S}(g'(Lh(u)+\Phi_0-u)d,Lh(u)+\Phi_0-u))\\
  &= -L(\underbrace{\mathcal{S}(g'(\Phi(u)-u)d,\Phi(u)-u)}_{=:\delta})
\end{align*}
where $\delta$ satisfies (by definition of $\mathcal{S}$)
\begin{equation*}
(k-\Delta)\delta - g'(\Phi(u)-u)L\delta = g'(\Phi(u)-u)d.
\end{equation*}
\end{proof}
The following estimate is immediate thanks to the local increasing property of $L$:
\begin{align}
\min(1,k)\norm{\delta}{W} \leq \norm{g'}{\infty}\norm{d}{H},\label{eq:deltaBd}
\end{align}
and by linearity it follows that solutions of the above PDE are unique.
\begin{corollary}\label{cor:allAssumptions}
Let \ref{item:smallnessForF}  hold. Then assumptions \ref{itm:PhiDiff}, \ref{itm:compContOfPhi}, \ref{itm:compContOfDerivOfPhi}, \ref{item:assPhiDerivativeBounded} are satisfied. If also
\[\min(1,k)^{-1}\norm{L}{\mathcal{L}(W,V)}\norm{g'}{\infty} < \frac 12,\]
then \ref{itm:smallnessOfDerivOfPhi} is satisfied.
\end{corollary}
\begin{proof}
It remains for us to show the latter two assumptions. Let us see why the mapping $d \mapsto \Phi'(u)(d)$ is completely continuous. Let $d_n \weaklyto d$ in $V$. Using the continuous dependence formula above, we find
\[\min(1,k)\norm{\delta_n-\delta_m}{W} \leq \norm{g'}{\infty}\norm{d_n-d_m}{H}\]
and thus $\delta_n$ converges strongly in $W$; the limit can be easily identified as the correct one.

Take $b \in V$ and $h\colon (0,T) \to V$ a higher order term. Then by boundedness of $L$ and the estimate \eqref{eq:deltaBd},
\begin{align*}
&\norm{\Phi'(u+tb+\lambda h(t))h(t)}{V}\\
 &=  \lVert L(\mathcal{S}(g'(\Phi(u+tb+\lambda h(t))-u-tb-\lambda h(t))h(t),\Phi(u+tb+\lambda h(t))\\
 &\qquad\qquad\qquad\qquad\qquad\qquad\qquad\qquad\qquad\qquad\qquad\qquad-u-tb-\lambda h(t))\rVert_{V}\\
&\leq C \lVert\mathcal{S}(g'(\Phi(u+tb+\lambda h(t))-u-tb-\lambda h(t))h(t),\Phi(u+tb+\lambda h(t))\\
&\qquad\qquad\qquad\qquad\qquad\qquad\qquad\qquad\qquad\qquad\qquad\qquad-u-tb-\lambda h(t)\rVert_W\\
& \leq C\min(1,k)^{-1}\norm{g'}{\infty}\norm{h(t)}{H}
\end{align*}
which vanishes in the limit after division of $t$. Thus assumption \ref{item:assPhiDerivativeBounded} is satisfied.

By the previous theorem, we see that
\begin{align*}
\norm{\Phi'(z)(d)}{V} &\leq \norm{L}{\mathcal{L}(W,V)}\norm{\delta}{W}\\
&\leq \min(1,k)^{-1}\norm{L}{\mathcal{L}(W,V)}\norm{g'}{\infty}\norm{d}{H}
\end{align*}
which by assumption leads to \ref{itm:smallnessOfDerivOfPhi}.
\end{proof}
\subsection{Numerical results}
We consider the model \eqref{eq:QVIforu}--\eqref{eq:mould} in two dimensions on the domain $[0,1]\times [0,1]$ with homogeneous Dirichlet conditions for the QVI. We approximate the QVI \eqref{eq:QVIforu} by a penalised equation and numerically solve the system
\begin{equation}\label{eq:approximatedSystem}
\begin{aligned}
Au + \alpha\max(0,u - y) -f &= 0\\
kT-\Delta T -g(y-u)&= 0\\
\partial_\nu T &= 0\\
y-\Phi_0 -LT &= 0
\end{aligned}
\end{equation}
for a parameter $\alpha$ large (as $\alpha \to \infty$, the solution of \eqref{eq:approximatedSystem} converges to the solution of \eqref{eq:QVIforu}--\eqref{eq:mould}). The penalisation of the QVI as a PDE was inspired by the case for variational inequalities (eg. see \cite{MR2822818}) and will be discussed in a future work by the authors. We use a finite difference scheme with $N^2$ uniformly distributed nodes and meshsize $h=1/(N+1)$ with $N=256$. The system \eqref{eq:approximatedSystem} is discretised and solved via a semismooth Newton method applied to the mapping 
\[\mathcal{F}\colon V \times W \times V \to V^* \times W^* \times V, \qquad (u,T,y) \mapsto \mathcal{F}(u,T,y)\] defined by the left hand side of \eqref{eq:approximatedSystem}. For this purpose, the derivative
\[\mathcal{F}'(u,T,y)(v,\tau,z) = \begin{pmatrix}
Av + \alpha\max'(0,u-y)(v-z)\\
k\tau-\Delta\tau - g'(y-u)(z-v)\\
\partial_\nu \tau \\
z-L\tau
\end{pmatrix}\]
was needed in order to obtain the next iterate through the Newton update scheme. Here, $\max'(0,u-y)$ denotes the Newton derivative of the maximum function, given by \S 8.3 in \cite{Ito2008}
\[{\max}'(0,u) = 
\begin{cases}
0 &\text{if $u < 0$}\\
\delta_N &\text{if $u=0$}\\
1 &\text{if $u > 0$}
\end{cases}
\]
where $\delta_N \in [0,1]$ is arbitrary; for the computation we pick $\delta_N = 0.1$. An approximation to the directional derivative of the QVI solution mapping was computed by first smoothing the nonlinearity in the first equation of \eqref{eq:approximatedSystem} by a function $\max_g$ (which is the global smoothing used in \cite{MR2822818} where the smoothing parameter is $10^{-5}$), and differentiating it with respect to $f$ in a direction $d$:
\[Au'(f)(d) +\alpha{\max}_g'(0,u-y)u'(f)(d) - d= 0,\]
i.e., it is related to the solution of the PDE
\begin{equation}\label{eq:pdeForDer}
Aw + \alpha{\max}_g'(0,u-y)w -d =0
\end{equation}
for a direction $d$. This equation was solved 
and was checked to be within a tolerance of $10^{-4}$ of the following object
\[\frac{u(f+ \epsilon d)-u(f)}{\epsilon }\]
with $\epsilon =10^{-5},$ which is an approximation to the difference quotient definition of the directional derivative.

\subsubsection{Choice of terms and initial iterate}
For the source term $f$, we choose the constant function $f\equiv 10^2$. The nonlinearity $g$ appearing in the source term for the $T$ equation is selected as the following smoothing of \eqref{eq:exampleg} for two parameters $\kappa>0$ and $s > 0$:
\begin{equation}\label{eq:numericsg}
g(r) = \begin{cases}
	   \kappa &\text{if $r  \leq 0$}\\    
       \kappa-{8\kappa r^2 }\slash{3s^2}&\text{if $0 < r \leq s/4$}\\          
        7\kappa/6 - 4\kappa r/3s &\text{if $s/4 < r \leq 3s/4$}\\
8\kappa(s-r)^2/3s^2 &\text{$3s/4 < r \leq s$}\\
 0 &\text{$r \geq s$}
 \end{cases},
\end{equation}
see Figure \ref{fig:g}.
\begin{figure}[H]
  \centering
        \includegraphics[width=0.6\textwidth]{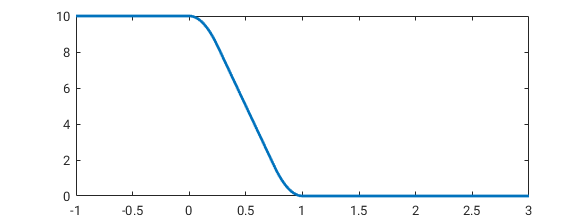}\\
                \includegraphics[width=0.6\textwidth]{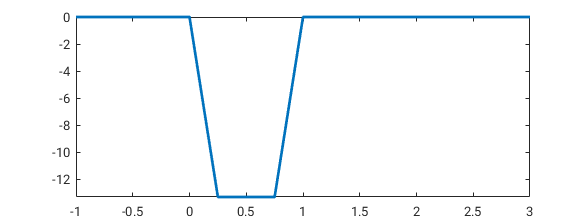}\\
      \caption{Plot of the function $g$ in \eqref{eq:numericsg} and its derivative, with $\kappa=10$ and $s=1$}\label{fig:g}
\end{figure}
The operator $L$ is chosen as the superposition mapping
\[(Lv)(x) = 5.25\times 10^{-3}\rho(x)v(x) =: C_L\rho(x)v(x)\]
where $\rho\colon [0,1]\times [0,1]\to \mathbb{R}$ is a smooth bump function (constructed from the exponential function) which is zero on the boundary with $\norm{\rho}{\infty}=1$ and $\norm{\grad \rho}{\infty}\leq \sqrt{50}$. Let us check the condition of Corollary \ref{cor:allAssumptions} that assures \ref{itm:smallnessOfDerivOfPhi}. We see that, given our choice of $g$ and $k$,
\begin{align*}
\min(1,k)^{-1}\norm{L}{\mathcal{L}(W,V)}\norm{g'}{\infty}&= \frac{40}{3}\norm{L}{\mathcal{L}(W,V)}
\end{align*}
where the operator norm on the right hand side can be estimated by the calculation 
\begin{align*}
\norm{Lv}{V}^2 &= \int |C_L\varphi|^2 |v|^2 + |\nabla (C_L\varphi v)|^2\\
&\leq C_L^2\norm{\varphi}{\infty}^2 \int v^2 + |\nabla v|^2 + C_L^2\norm{\nabla \varphi}{\infty}^2\int v^2\\
&\leq C_L^2(\norm{\varphi}{\infty}^2 + \norm{\nabla \varphi}{\infty}^2)\norm{v}{W}^2,
\end{align*}
so that
\begin{align*}
\min(1,k)^{-1}\norm{L}{\mathcal{L}(W,V)}\norm{g'}{\infty} &\leq 
\frac{40C_L\sqrt{51}}{3} 
\end{align*}
and hence if 
\[C_L < \frac{3}{80\sqrt{51}}\approx 0.00525,\]
we have assumption \ref{itm:smallnessOfDerivOfPhi}. 

        
As for the initial mould $\Phi_0$, we choose it as follows. Define $w\colon [0,1] \to \mathbb{R}$        
\begin{equation}
w(r)=
\begin{cases}
5(r/N-1/10) &\text{if $N/10 \leq r \leq 3N/10$}\\
1 &\text{if $3N/10 <r < 7N/10$}\\
1- 5(r/N-7/10) &\text{if $7N/10 \leq r \leq 9N/10$}\\
0 &\text{otherwise}
\end{cases}
\end{equation}
and set $\Phi_0(r,t) = w(r)w(t)$, see Figure \ref{fig:initialMould}.
\begin{figure}[H]
  \centering
        \includegraphics[width=0.7\textwidth]{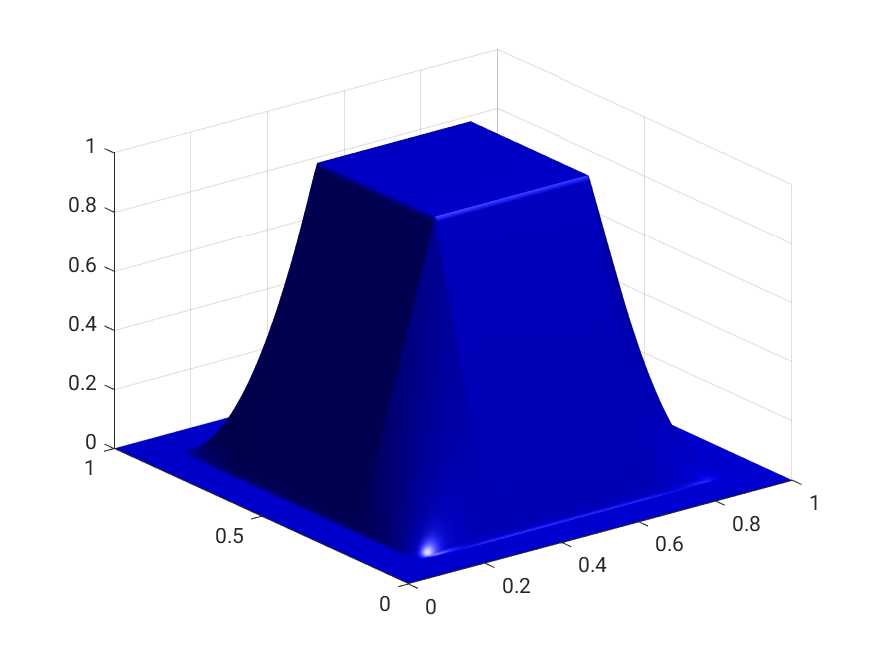}
      \caption{The initial mould $\Phi_0$}\label{fig:initialMould}
\end{figure}
The directional derivative was taken in the direction $\chi_A$, the characteristic function of the set $A=\{(x,y) : x > 1/2\}$.
The remaining parameters appearing in the physical model are     $k =1$, $\alpha=10^8$, $\kappa = 10$ and $s=1$. The initial iterates for $(u^0_h,T^0_h,y^0_h) = (0.9\times \Phi_0, 0.2, 10)$ were used, where $\Phi_0$ is as defined above through the bump function.
\subsubsection{Numerical details}

The Newton iterates $(u^j_h,T^j_h,y^j_h)$, were assumed to converge if  $\|\mathcal{F}(u^j_h,T^j_h,y^j_h)\|_{L^2}<4\times10^{-9}$, for some $j$, and we denote the solution as $(u_h,T_h,y_h):=(u^j_h,T^j_h,y^j_h)$. Note that the $L^2$-norm is a stronger choice than the one for $V^*\times W^*\times V$, further improving accuracy for $\mathcal{F}(u,T,y)=0$.


\begin{table}
\centering
    \begin{tabular}{| l | p{3cm} | p{3cm}| p{3cm}| p{3cm}|}
    \hline
    \textbf{No. of nodes} & \# \textbf{Newton iterations to solve system \eqref{eq:approximatedSystem}} & \textbf{$L^2$ error in solution of system} & \textbf{$L^2$ error in solution of derivative}\\ \hline
    256& 14  &$3.96\times 10^{-9}$ &$1.96\times 10^{-15}$\\ \hline

   \end{tabular}
   \caption{Numerical results}\label{table:numericalResults}
\end{table}
\subsubsection{Results and analysis}
See Table \ref{table:numericalResults} for the numerical results. The third column in the table refers to the error of the approximate solution $(u_h, T_h, y_h)$ in that it measures the $L^2$ norm of $\mathcal{F}(u_h, T_h, y_h)$, and likewise for the fourth column. One can see that a relatively low number of Newton iterations is performed to obtain an accurate solution. The results of the experiment are visualised in Figure \ref{fig:results256}. Let us highlight some interesting observations.
\begin{itemize}
\item The effect of the temperature interplay between the membrane and mould can be immediately seen: the initial mould $\Phi_0$ (Figure \ref{fig:initialMould}) grows and becomes more curved and smoothed out, which is natural given that the membrane is initially placed below the mould and is pushed upwards.
\item The model produces a membrane $u$ that appears to be rather a good fit for the thermoforming process; it can be observed to be similar to the final mould, which is confirmed by the images of the coincidence sets
\item The directional derivative is coloured yellow and red; red refers to the parts of the domain corresponding to the coincidence set $\{u=y\}$. 
\end{itemize}

\begin{figure}[H]
  \centering
        \begin{subfigure}[b]{0.3\textwidth}
        \includegraphics[width=\textwidth]{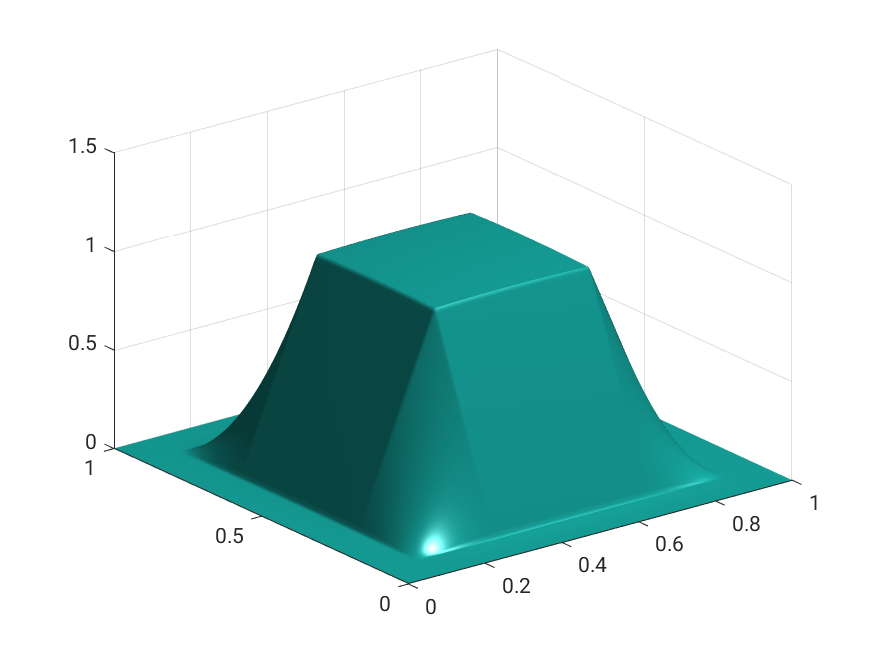}
        \caption{Final mould $y$}
    \end{subfigure}
      \begin{subfigure}[b]{0.3\textwidth}
        \includegraphics[width=\textwidth]{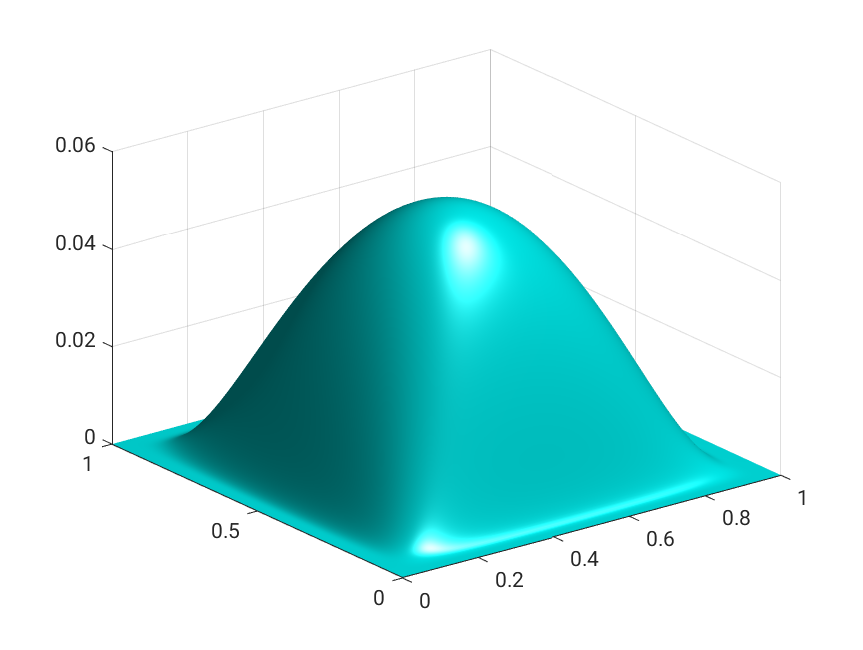}
        \caption{Difference between $y$ and $\Phi_0$}

    \end{subfigure}
      \begin{subfigure}[b]{0.3\textwidth}
        \includegraphics[width=\textwidth]{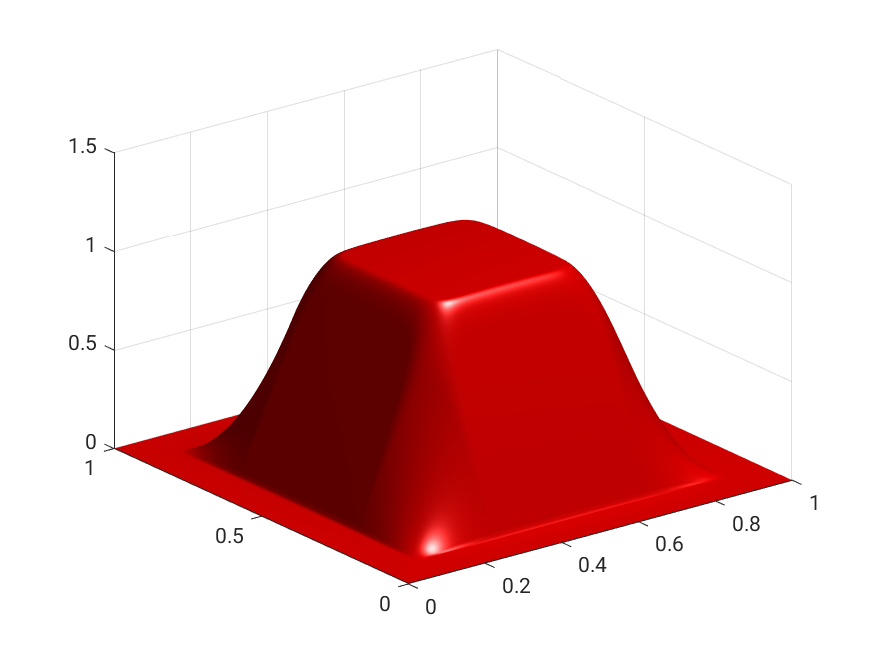}
        \caption{Membrane $u$}

    \end{subfigure}
          \begin{subfigure}[b]{0.3\textwidth}
        \includegraphics[width=\textwidth]{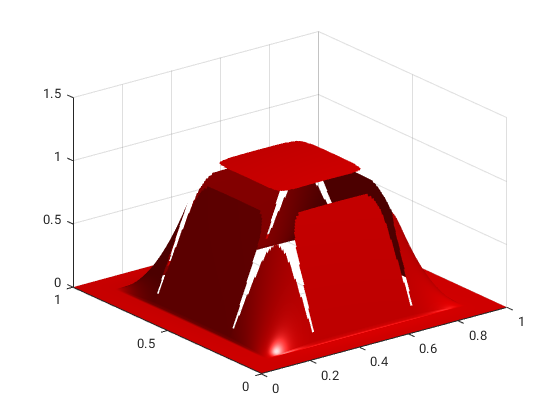}
        \caption{Membrane $u$ on the coincidence set}

    \end{subfigure}
      \begin{subfigure}[b]{0.3\textwidth}
        \includegraphics[width=\textwidth]{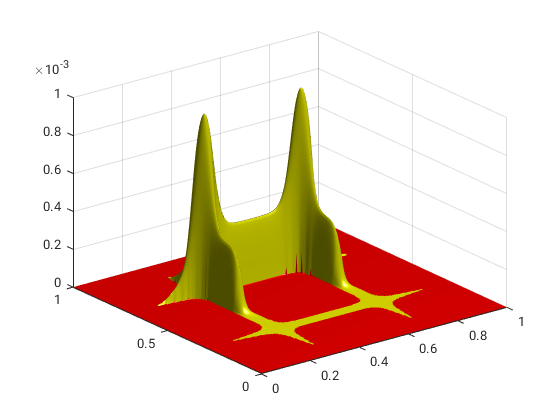}
        \caption{The directional derivative}
    \end{subfigure}    
      \caption{Computation results}\label{fig:results256}
\end{figure}
Figure \ref{fig:results256CASc} shows the coincidence set.


\begin{figure}[H]
  \centering
        \includegraphics[width=\textwidth]{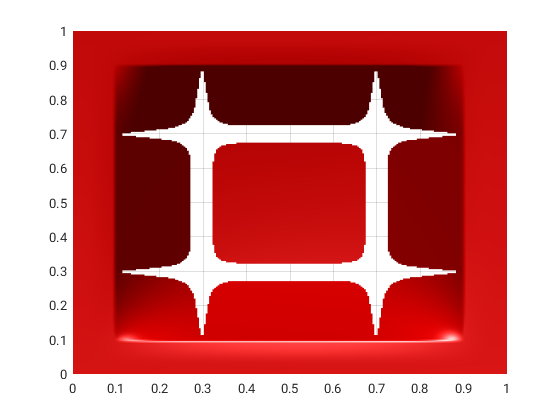}
      \caption{Top-down view of the coincidence set}\label{fig:results256CASc}
\end{figure}

\section{Further work}
We did not fully exploit the monotonicity results in \S \ref{sec:properties} and we aim to address this in a future work, where the removal of the inconvenient assumption \ref{itm:smallnessOfDerivOfPhi} would be desirable. In relation to this, we could consider a weaker notion of directional derivative: instead of asking the derivative to exist in $V$ we could simply ask for it to exist in the $L^2(X)$ topology instead. The numerical experiments on the thermoforming application in \S \ref{sec:thermoforming} can be further solidified by verification of strict complementarity or lack thereof and it would also be interesting to study the properties of the computed directional derivative with respect to the characterisation obtained in Theorem \ref{thm:main}. The optimal control of QVI is a work in progress by the current authors, and numerical analysis related to these subjects will also be explored.
\section*{Acknowledgements}
This research was carried out in the framework of \textsc{MATHEON} supported by the Einstein Foundation Berlin within the
ECMath projects OT1, SE5, CH12 and SE15/SE19 as well as project A-AP24. The authors further acknowledge the
support of the DFG through the DFG-SPP 1962: Priority Programme ``Non-smooth and Complementarity-based
Distributed Parameter Systems: Simulation and Hierarchical Optimization''  within Projects 10, 11, and 13, through grant no. HI 1466/7-1 Free Boundary Problems and Level Set Methods, and SFB/TRR154. The authors wish to thank Gerd Wachsmuth and Daniel Wachsmuth for their useful comments which helped improve the work.
\appendix
\section{Simplification of the Laplace--Beltrami operator}\label{sec:curvature}
Define $T(r) := \hat T(r,\Phi(u)(r))$ so that $T\colon [0,1]\to \mathbb{R}$. 
We want to write $\Delta_\Gamma \hat T$ in terms of $\Delta T$. With $w:=\Phi(u)$, the metric tensor is $g_{11}=g = 1+ (w')^2$ and its inverse is $g^{11}={1}\slash(1+(w')^2)=g^{-1}$ and so the Laplace--Beltrami of a function $\hat H\colon \Gamma \to \mathbb{R}$ can be written as
\[\Delta_\Gamma \hat H = \frac{1}{\sqrt{g}}(g^{11}\sqrt{g}H')'\] where $H\colon [0,L] \to \mathbb{R}$ is defined by $H(r) = \hat H(r,w(r))$.
Thus
\begin{align*}
\Delta_\Gamma \hat H 
=\frac{H''}{g} - \frac{H'g'}{2g^2}.
\end{align*}
Noticing that
$g' = 2w'w''$, the above reads
\begin{align*}
\Delta_\Gamma \hat H 
&=\frac{H''}{1+(w')^2} - \frac{w'w''H'}{(1+(w')^2)^2}
\end{align*}
And now picking $\hat H=\hat T$, we find
\begin{align*}
\Delta_\Gamma \hat T
&=\frac{\Delta T}{1+(\Phi(u)')^2} - \frac{T'\Phi(u)'\Delta \Phi(u)}{(1+(\Phi(u)')^2)^2}
\end{align*}
Then if $x=(x_1,x_2) = (r, \Phi(r))$ the equation \eqref{eq:pdeForhatT} becomes
\begin{align*}
kT(r) - \frac{\Delta T(r)}{1+(\Phi(u)'(r))^2} + \frac{T'(r)\Phi(u)'(r)\Delta \Phi(u)(r)}{(1+(\Phi(u)'(r))^2)^2} 
&= g(\Phi(r)-u(r)).
\end{align*}
In order to deduce \eqref{eq:pdeForT} we have taken $\Phi(u)'$ to be close to zero in the equality above.

\bibliographystyle{abbrv}
\bibliography{QVIPaper}


\end{document}